\documentclass[a4paper]{article}

\usepackage[T1]{fontenc}
\usepackage{fullpage}
\usepackage{graphicx}
\usepackage[utf8]{inputenc}
\usepackage[british]{babel}
\usepackage{setspace}
\usepackage{amsmath}
\usepackage{mathrsfs}
\usepackage{amssymb}
\usepackage{amsthm}
\usepackage{booktabs}
\usepackage{caption}
\usepackage{tikz-cd}
\usepackage{float}
\usepackage{rotating}
\usepackage{verbatim}
\usepackage{epigraph}
\usepackage{geometry}
\usepackage{datetime}
\usepackage{youngtab}
\usepackage{multirow}
\usepackage{hyperref}
\usepackage[normalem]{ulem}
\usepackage{makecell}

\usepackage{enumitem}
\setlist{nolistsep,noitemsep}

\theoremstyle{plain} \newtheorem{thm}{Theorem}[section]
\theoremstyle{definition} \newtheorem{defn}[thm]{Definition}
\theoremstyle{plain} \newtheorem{prop}[thm]{Proposition}
\theoremstyle{plain} \newtheorem{lem}[thm]{Lemma}
\theoremstyle{plain} 
\theoremstyle{remark} \newtheorem{rk}[thm]{Remark}
\theoremstyle{plain} \newtheorem{cj}[thm]{Conjecture}
\theoremstyle{remark} 
\theoremstyle{remark} 
\theoremstyle{definition} 
\theoremstyle{definition} \newtheorem{eg}[thm]{Example}
\theoremstyle{definition} 

\newcommand{\numberset}{\mathbb}
\newcommand{\N}{\numberset{N}}

\newcommand{\C}{\numberset{C}}
\newcommand{\R}{\numberset{R}}
\newcommand{\Z}{\numberset{Z}}
\newcommand{\poids}{\mathtt{P}}

\newcommand{\lieg}{\mathfrak{g}}

\newcommand{\orb}{\mathcal{O}}
\newcommand{\Orb}{\mathfrak{O}}

\DeclareMathOperator{\bir}{Bir}
\DeclareMathOperator{\ind}{Ind}

\DeclareMathOperator{\Ad}{Ad}

\DeclareMathOperator{\Lie}{Lie}

\DeclareMathOperator{\SO}{SO}
\DeclareMathOperator{\SL}{SL}

\title{Spherical birational sheets in reductive groups}
\author{Filippo Ambrosio, Mauro Costantini\\
Dipartimento di Matematica ``Tullio Levi-Civita''\\
Torre Archimede - via Trieste 63 - 35121 Padova - Italy\\
ambrosio@math.unipd.it, costantini@math.unipd.it}
\date{}

\begin{document}
\maketitle
\begin{abstract}
We classify the spherical birational sheets in a complex simple simply-connected algebraic group.
We use the classification to show that, when $G$ is a connected reductive complex algebraic group with simply-connected derived subgroup, two conjugacy classes $\orb_1$, $\orb_2$ of $G$ lie in the same birational sheet, up to a shift by a central element of $G$, if and only if the coordinate rings of $\orb_1$ and $\orb_2$ are isomorphic as $G$-modules. As a consequence, we prove a conjecture of Losev for the spherical subvariety of the Lie algebra of $G$.
\end{abstract}

MSC-class:  20G20 (Primary) 14M27 (Secondary)


\section{Introduction}
Let $G$ be a complex connected reductive algebraic group acting on a variety $X$.
A sheet of $X$ is an irreducible   
component of the locally closed subset $\{x \in X \mid \dim (G \cdot x) = d\}$ for some fixed $d$: then $X$ is the finite union of its sheets. 
Let $B$ be a Borel subgroup of $G$, the complexity of $X$ is the codimension of a generic $B$-orbit in $X$.
The variety $X$ is \emph{spherical} if has complexity zero.
By \cite[Proposition 1]{Arzhantsev_1997},  the complexity of orbits as homogeneous spaces of $G$ is constant along the sheets.
In particular it follows that the property of being spherical is preserved along sheets. 
We say that the sheet $S$ is  \emph{spherical} if the orbits in $S$ are spherical.
Now assume $X=G$, and the action is given by conjugation.
Let $T$ be a maximal torus of $B$, with Weyl group $W$.
From the Bruhat decomposition $G=\bigcup_{w\in W}BwB$, it follows that for every conjugacy class $\orb$ of $G$ there exists a unique $w_{\orb} \in W$ such that $\orb \cap B w_{\orb} B$ is dense in $\orb$.
Similarly, for $S$ a sheet of conjugacy classes, there is a unique  $w_S \in W$ such that $S \cap B w_S B$ is dense in $S$.
By \cite[Proposition 5.3]{CarnoBul} if $S$ is a spherical sheet, then for every conjugacy class $\orb$ lying in $S$ we have $w_{\orb}=w_S$.

A natural question is to consider the ring of regular functions $\C[\orb]$ as $\orb$ varies in a sheet $S$ and ask whether the $G$-modules $\C[\orb]$ are isomorphic.
When $G$ acts via the adjoint action on its Lie algebra $\lieg$, some answers were obtained in \cite{BK79}:
for $\lieg = \mathfrak{sl}_n(\C)$  the $G$-module structure of $\C[\orb]$ is preserved along sheets, but in this fails in general.
In \cite{Losev}, Losev refined the notion of sheets of adjoint orbits by introducing the definition of birational sheets. In \cite[Theorem 4.4]{Losev}, it is proven that birational sheets are locally closed subvarieties partitioning $\mathfrak{g}$.
A remarkable result (see \cite[Remark 4.11]{Losev}) states that 
if $\Orb_1$ and $\Orb_2$ are adjoint orbits of $\mathfrak{g}$ lying in the same birational sheet, then their $G$-module structures are isomorphic.
In the same Remark, Losev conjectured that the viceversa is also true, aiming for an intrinsic characterization of birational sheets of the Lie algebra.

In this paper we deal with this problem with respect to spherical orbits both in the setting of conjugacy classes in $G$ with simply-connected  derived subgroup and in the setting  of adjoint orbits in $\lieg$.
We recall the definition of birational sheet in $\lieg$ from \cite{Losev} and in $G$ from \cite{AF}.
A birational sheet is a certain union of $G$-orbits and is contained in a 
sheet, hence the property of being spherical is preserved along birational sheets. 
We shall call \emph{spherical birational sheet} any birational sheet consisting of spherical orbits.
For $G$ simple simply-connected, we classify the spherical birational sheets and  observe 
 that the union $G_{sph}$ of all spherical conjugacy classes in $G$ is the disjoint union of spherical birational sheets.
If $\orb$ is a spherical conjugacy class, then $\C[\orb]$ is multiplicity-free, i.e. a simple $G$-module occurs in $\C[\orb]$  with multiplicity at most $1$.
Therefore, $\C[\orb]$ is completely determined as a $G$-module by its weight monoid, i.e. by the highest dominant weights $\lambda$ for which the simple $G$-module  with highest weight $\lambda$ occurs in the decomposition of $\C[\orb]$. 

In \cite{Costantini2010} the weight monoids  are explicitely described for every spherical conjugacy class of $G$ simple simply-connected.
Using these results and the classification of spherical birational sheets, we shall prove the main result of this paper: let $G$ be a complex connected reductive algebraic group with simply-connected derived subgroup and let $\orb_1$ and $\orb_2$ be spherical conjugacy classes in $G$.
Let $S_1^{bir}$ (resp. $S_2^{bir}$) be the birational sheet containing $\orb_1$ (resp. $\orb_2$).
Then $\C[\orb_1]$ is isomorphic to $\C[\orb_2]$ as a $G$-module if and only if $S_2^{bir}= zS_1^{bir}$ for some $z \in Z(G)$ (the assumption on the derived subgroup of $G$ cannot be relaxed).

From this we also deduce the validity of Losev's conjecture in the case of spherical adjoint orbits in $\lieg$. 
We also show that Losev's conjecture (resp. the corresponding group anologue) is true in the case $\lieg =
{\mathfrak{sl}}_n(\C)$ (resp. $G=\SL_n(\C)$).

\section{Definitions and notations}
Let $G$ be a connected reductive algebraic group over $\C$
 and let $\mathfrak{g}$ be its Lie algebra.
If $K$ is a closed subgroup of $G$, we denote by $K^\circ$ its identity component, by $K'$ its derived subgroup and by $Z(K)$ its centre.
Similarly, if $\mathfrak{k}$ is a Lie subalgebra of $\mathfrak{g}$, we denote by $\mathfrak{z(k)}$ its centre.

If $X$ is a $K$-set, we denote by $X/K$ the set of $K$-orbits of elements in $X$.
When $K$ acts regularly on a variety $X$ and $x \in X$, the $K$-orbit of $x$ is denoted by $K \cdot x$.
For any $n \in \N$, we define the locally closed subsets $X_{(n)} := \{ x \in X \, \mid \, \dim (K \cdot x) = n\}$ of $X$.
A \emph{sheet} of $X$ for the action of $K$ is an irreducible component of $X_{(n)}$ for some $n \in \N$ such that $X_{(n)} \neq \varnothing$.
For $Y \subseteq X$, the regular locus of $Y$ is $Y^{reg} = Y \cap X_{(\bar{n})}$, where $\bar{n} = \max \{n \in \N \, \mid \, Y \cap X_{(n)} \neq \varnothing \}$, an open subset of $Y$, and the normalizer of $Y$ in $K$ is $N_K(Y) :=\{ k \in K \, \mid \, k \cdot Y=Y\}$.
For $x \in X$, its stabilizer is $K_x := \{ k \in K \, \mid \, k \cdot x = x\}$.
When we consider the conjugacy (resp. the adjoint) action of $G$ on itself (resp. on $\mathfrak{g}$) we adopt the following notation for orbits and stabilizers.
Dealing with $K$-conjugacy classes or $K$-adjoint orbits, we shall use the notation $\orb^K_g := K \cdot g$, 
 $\Orb^K_\xi:= \Ad(K)(\xi)$.
We shall omit superscripts whenever $K=G$.
For  $x \in G$ and $\eta \in \mathfrak{g}$, we write:
\begin{align*}
C_G(x) &:= G_x = \{ g \in G \, \mid \, gxg^{-1} = x \};  &C_G(\eta) &:= G_\eta = \{ g \in G \, \mid \, (\Ad g)(\eta) = \eta \}; \\
\mathfrak{c_g}(x) &:= \{\xi \in \mathfrak{g} \, \mid \, (\Ad x)(\xi) = \xi \}; & \mathfrak{c_g}(\eta) &:= \mathfrak{g}_\eta  = \{ \xi \in \mathfrak{g} \, \mid \, [\eta, \xi] = 0 \}.
\end{align*}
For a subset $Y \subseteq G$, we set $C_G(Y):= \bigcap_{y \in Y} C_G(y)$.

We write $\mathcal{U}_K$ 
for the unipotent variety of $K$ and $\mathcal{N}_{\mathfrak{k}}$ 
for the nilpotent cone of $\mathfrak{k} := \Lie(K)$; we also set $\mathcal{U} := \mathcal{U}_G$ and $\mathcal{N} := \mathcal{N}_{\mathfrak{g}}$.
The set of all $K$-conjugacy classes of $K$ is denoted $K/K$.

When we write $g = su\in G$ we implicitly assume that $su$ is the Jordan decomposition of $g$, with $s$ semisimple and $u$ unipotent. Similarly for $\xi = \xi_s + \xi_n \in \mathfrak{g}$.

Let $B$ be a Borel subgroup of $G$ and $T$ a maximal torus of $B$. 
We denote by $\Phi$ the root system of $G$ with respect to $T$, by $\Delta$ the base of $\Phi$ individuated by $B$ and by $\Phi^+$ the corresponding subset of positive roots.
The one-parameter subgroup of $G$ corresponding to the root $\alpha \in \Phi$ will be denoted by $U_\alpha$.
We call \emph{Levi subgroup} of $G$ every Levi factor of a parabolic subgroup of $G$. 

A standard parabolic subgroup is a subgroup containing $B$: it is of the form $P_\Theta = \langle B, U_{-\alpha} \mid \alpha \in \Theta \rangle$ for $\Theta \subseteq \Delta$. We have $P_\Theta = L_\Theta U_\Theta$, where the Levi factor $L_\Theta := \langle T, U_\alpha, U_{-\alpha} \mid \alpha \in \Theta \rangle$ is called a \emph{standard Levi subgroup} and $U_\Theta$ is the unipotent radical of $P_\Theta$.  We also set $\Lie(T) = \mathfrak{h}$, $\Lie(B) = \mathfrak{b}$, $\Lie(U_\alpha) = \mathfrak{g}_\alpha$ for all $\alpha \in \Phi$.

A \emph{pseudo-Levi subgroup} is the connected centralizer of a semisimple elements of $G$.

Finite-dimensional irreducible $G$-modules are parametrized by $X(T)^+$, the set of dominant weights of $T$ (with respect to $\Phi^+$), and we write  $V(\lambda)$ for the irreducible $G$-module of highest weight $\lambda$.

Let $X$ be a conjugacy class in $G/G$ or an adjoint orbit in $\mathfrak{g}/G$. We have a decomposition into simple $G$-modules of the ring of regular functions $\C[X]$:
 $$ \C[X]~\simeq_G~\bigoplus\limits_{\lambda \in X(T)^+} n_\lambda V(\lambda),$$
    where $n_\lambda$  is the \emph{multiplicity} with which $V(\lambda)$ occurs in $\C[X]$, denoted by $[\C[X]:V(\lambda)]$:  we denote by $\Lambda(X)$ the monoid of dominant weights occurring in $\C[X]$.
If a Borel subgroup of $G$ has a dense orbit on $X$, we call $X$ \emph{spherical}.
Since $X$ is quasi-affine, this is equivalent to the fact that $\C[X]$ is \emph{multiplicity-free}, i.e. $n_\lambda \in \{0,1\}$ for every $\lambda\in X(T)^+$: hence
$\C[X] \simeq_G \bigoplus_{\lambda\in \Lambda(X)} V(\lambda)$.

A closed subgroup $H \leq G$ is said to be \emph{spherical} if the homogeneous space $G/H$ is a spherical variety.

We denote by $G_{sph}$ (resp. $\mathfrak{g}_{sph}$) the union of all spherical conjugacy classes in $G$ (resp. spherical adjoint orbits in $\mathfrak{g}$): these are closed  subsets by \cite[Corollary 2]{Arzhantsev_1997}.

When $G$ is simple, we denote the simple roots by
 $\alpha_1, \dots, \alpha_n$:
 we shall use the numbering and the
description of the simple roots in terms of the canonical basis $(e_1,\ldots,e_k)$ of an appropriate
$\R^k$ as in \cite[Planches I–IX]{Bour4}.  We denote by $\poids$ the weight lattice, by $\poids^+$ the monoid of
dominant weights. 
Also, 
$\alpha^\vee_1, \dots, \alpha^\vee_n$ are the co-roots,
$\omega_1,\ldots\omega_n$ are the fundamental weights and $\check\omega_1,\ldots\check\omega_n$ are the fundamental co-weights: these are the elements $\check{\omega}_j$ of $\mathfrak{h}$ defined by 
$\alpha_i (\check{\omega}_j) = \delta_{ij}$ for $1 \leq i,j \leq n$.
The Weyl group of $G$ is denoted by $W$, for $w \in W$ we use the notation $\dot{w}$ for an element of $N_G(T)$ representing $w \in W \simeq N_G(T)/T$. We write $s_i$ for the simple reflection with respect to the simple root $\alpha_i$, for $i =1, \dots, n$.
Let $\beta=\sum_{j=1}^nc_j\alpha_j$ be the highest root in $\Phi$: we define $\widetilde{\Delta} = \Delta \cup \{-\beta \}$.
For the exceptional groups, we shall write $\beta=(c_1,\dots,c_n)$ 
For $\Theta \subsetneq \widetilde{\Delta}$, set $L_\Theta := \langle T, U_\alpha, U_{-\alpha} \mid \alpha \in \Theta \rangle$.
Following the terminology introduced in \cite{SommersBC}, we say that $L_\Theta$ is a \textit{standard pseudo-Levi subgroup} of $G$.
By \cite[Proposition 2]{SommersBC}, pseudo-Levi subgroups are conjugates of standard pseudo-Levi subgroups. 

An element $su \in G$ is \emph{isolated} if  $C_G(Z(C_G(s)^\circ)^\circ) = G$, as in \cite[Definition 2.6]{LusztigICC}; in this case we say that $\orb_{su}$ is an \emph{isolated class}.

A partition of $n \in \N \setminus \{0\}$  is a sequence of non-increasing positive integers $\mathbf{d} = [d_1, \dots, d_r] $ such that $\sum_{i=1}^r d_i = n$: we write $\mathbf{d} = [d_1, \dots, d_r] \vdash n$. If $\mathbf{d} \vdash n$, the  dual partition is $\mathbf{d}^{t} = \mathbf{f}$, where $f_i = \lvert \{j \mid d_j \geq i\} \rvert$ for all $i$. 
  We will also use the compact notation $\mathbf{d} = [e_1^{m_1}, \dots, e_s^{m_s}]$
where $e_1 > \dots > e_s >0$ by grouping equal $d_i$'s.
Partitions will be used to denote nilpotent orbits in classical Lie algebras, whereas for exceptional Lie algebras we will use the Bala-Carter labeling, as in \cite{CollMcGov}.

We use the symbol $\sqcup$ to denote a disjoint union.
\section{Jordan classes, sheets and birational sheets}

\subsection{Lie algebra case}
Let $\mathfrak{l} \subseteq \mathfrak{g}$ be a Levi subalgebra and embed it in a parabolic subalgebra $\mathfrak{p} = \mathfrak{l} \oplus \mathfrak{n}$, where $\mathfrak{n}$ is the nilradical of $\mathfrak{p}$. Let $P\leq G$ such that $\Lie (P) = \mathfrak{p}$, and let $P = L U_P$ be its Levi decomposition with $\Lie(L) =\mathfrak{l}$ and $\Lie (U_P) = \mathfrak{n}$.
Let $\Orb^L \in \mathcal{N}_\mathfrak{l}/L$.
Then $P$ acts on the closed subvariety $\overline{\Orb^L} + \mathfrak{n} \subseteq \mathfrak{g}$ via the adjoint action.
The \textit{generalized Springer map} is:
\begin{equation}  \label{eq_gsm}
\gamma \colon  G \times^P (\overline{\Orb^L} + \mathfrak{n}) \to \Ad(G)  (\overline{\Orb^L} + \mathfrak{n}),\quad
 g * \xi \mapsto (\Ad g)(\xi).
 \end{equation} 
The image of $\gamma$ is the closure of a single orbit $\Orb \in \mathcal{N}/G$, and $\ind_\mathfrak{l}^\mathfrak{g} \Orb^L := \Orb$ is the orbit \textit{induced} from $\Orb^L$.
It only depends on the pair $(\mathfrak{l}, \Orb^L)$, not on the parabolic subgroup $P$ chosen to define \eqref{eq_gsm}.
If $\Orb \in \mathcal{N}/G$  cannot be induced from a nilpotent orbit $\Orb^L$ in a proper Levi subalgebra $\mathfrak{l} \subsetneq \mathfrak{g}$, then $\Orb$ is said to be \textit{rigid}.
For a complete exposition on induction, refer to \cite[\S 7]{CollMcGov}.

A \emph{decomposition datum}  of $\lieg$ consists of a pair of a Levi subalgebra $\mathfrak{l} \subseteq \mathfrak{g}$ and an orbit $\Orb^L \in \mathcal{N}_\mathfrak{l}/L$, see \cite[\S 1.6]{Borho}.
To any element $\xi=\xi_s + \xi_n \in \mathfrak{g}$ we can associate its decomposition datum $(\mathfrak{c_g}(\xi_s), \Orb^{C_G(\xi_s)}_{\xi_n})$.

We denote by $\mathscr{D}(\mathfrak{g})$ the set of decomposition data of $\mathfrak{g}$.
$G$ acts by simultaneous conjugacy on the elements of $\mathscr{D}(\mathfrak{g})$.
We say that two elements of $\mathfrak{g}$ are \emph{Jordan equivalent} if their decomposition data are conjugate in $G$.
The \emph{Jordan class} of $\xi \in \mathfrak{g}$ is the set $\mathfrak{J}(\xi)$ consisting of all elements which are Jordan equivalent to $\xi$.
If $\xi \in \mathfrak{g}$ has decomposition datum $(\mathfrak{l}, \Orb^L)$, then $\mathfrak{J}(\xi) = {\mathfrak{J}(\mathfrak{l}, \Orb^L)} =(\Ad G)(\mathfrak{z(l)}^{reg} + \Orb^L)$. Jordan classes form a partition of $\mathfrak{g}$ into finitely many irreducible subvarieties parametrized by the (finite) set $\mathscr{D}(\mathfrak{g})/G$.
They consist of unions of equidimensional adjoint orbits
and their closure $\overline{\mathfrak{J}(\mathfrak{l}, \Orb^L)}$ (resp. \emph{regular closure} $\overline{\mathfrak{J}(\mathfrak{l}, \Orb^L)}^{reg}$) is a union of Jordan classes.

Sheets for the adjoint action of $G$ on the Lie algebra $\mathfrak{g}$ have been studied in \cite{BK79, Borho}.
They are parametrized by the $G$-equivalence classes of decomposition data $(\mathfrak{l}, \Orb^L) \in \mathscr{D}(\mathfrak{g})$ such that $\Orb^L \in \mathcal{N}_{\mathfrak{l}}/L$ is rigid.
The sheet $\mathfrak{S}(\mathfrak{l}, \Orb^L)$ corresponding to the (class of) decomposition datum $(\mathfrak{l}, \Orb^L)$ with $\Orb^L$ is rigid is:
$$\mathfrak{S}(\mathfrak{l}, \Orb^L) = \overline{\mathfrak{J}(\mathfrak{l}, \Orb^L)}^{reg} = \bigcup_{\xi \in \mathfrak{z(l)}} (\Ad G) ( \xi + \ind_{\mathfrak{l}}^{\mathfrak{c_g}(\xi)} \Orb^L).$$
Every sheet $\mathfrak{S}(\mathfrak{l}, \Orb^L)$ contains a unique nilpotent orbit, i.e. $\ind_{\mathfrak{l}}^{\mathfrak{g}} \Orb^L$. The dimension of a sheet has been determined explicitly in \cite{Moreau1}, \cite{Moreau2}.

If  $\mathfrak{g}$ is simple of type $\mathsf{A}$, its sheets are disjoint and the $G$-module structure of the rings of functions $\C[\Orb]$ is preserved along sheets, see \cite{BK79}.
In general, these  properties do not hold and sheets intersect non-trivially.

In \cite{Losev}, Losev introduced birational sheets of $\lieg$ by restricting conditions on induction.
Let $(\mathfrak{l}, \Orb^L) \in \mathscr{D}(\mathfrak{g})$.
As in \cite[\S 4]{Losev}, we say that $\ind_\mathfrak{l}^\mathfrak{g} \Orb^L$ is \textit{birationally induced} from $(\mathfrak{l}, \Orb^L)$ if, for a (hence any) parabolic subalgebra $\mathfrak{p}$ with Levi factor $\mathfrak{l}$, the generalized Springer map as in \eqref{eq_gsm} is birational.
If $\Orb \in \mathcal{N}/G$ cannot be induced birationally from a proper Levi subalgebra, we say that $\Orb$ is \textit{birationally rigid};
all rigid orbits are birationally rigid.
For any $(\mathfrak{l}, \Orb^L) \in \mathscr{D}(\mathfrak{g})$, one can define, as in \cite[\S 4]{Losev}, the set 
$$\bir (\mathfrak{z(l)}, \Orb^L) = \{\xi \in \mathfrak{z(l)} \mid \ind_{\mathfrak{l}}^{\mathfrak{c_g}(\xi)} \Orb^L \mbox{ is birationally induced}\}.$$

Since $\Orb^L=\ind_\mathfrak{l}^\mathfrak{l} \Orb^L$ is birationally induced from $(\mathfrak{l}, \Orb^L)$, 
the inclusion $\mathfrak{z(l)}^{reg} \subset \bir (\mathfrak{z(l)}, \Orb^L)$ holds.
By \cite[Proposition 4.2]{Losev}, the set $\bir (\mathfrak{z(l)}, \Orb^L)$ is open in $\mathfrak{z(l)}$ and it is independent of the parabolic group chosen for induction.
For $(\mathfrak{l}, \Orb^L) \in \mathscr{D}(\mathfrak{g})$, the \emph{birational closure} of $\mathfrak{J}(\mathfrak{l}, \Orb^L)$ is defined by as follows:
$$\overline{\mathfrak{J}(\mathfrak{l}, \Orb^L)}^{bir}
= \bigcup_{\xi \in \bir(\mathfrak{z(l)}, \Orb^L)} (\Ad G) ( \xi + \ind_{\mathfrak{l}}^{\mathfrak{c_g}(\xi)} \Orb^L).$$
In particular $\overline{\mathfrak{J}(\mathfrak{l}, \Orb^L)}^{bir}$ is open in $\overline{\mathfrak{J}(\mathfrak{l}, \Orb^L)}^{reg}$ and in $\overline{\mathfrak{J}(\mathfrak{l}, \Orb^L)}$, hence it is irreducible and contained in a sheet.
\begin{defn}
For $(\mathfrak{l}, \Orb^L) \in \mathscr{D}(\mathfrak{g})$ with $\Orb^L$ birationally rigid, the \emph{birational sheet} corresponding to $(\mathfrak{l}, \Orb^L)$ is defined as $\overline{\mathfrak{J}(\mathfrak{l}, \Orb^L)}^{bir}$.
\end{defn}

In \cite[Theorem 4.4]{Losev}, it is proven that birational sheets are locally closed subvarieties partitioning the Lie algebra $\mathfrak{g}$; they are paramatrized by $G$-equivalence classes of pairs $(\mathfrak{l}, \Orb^L) \in \mathscr{D}(\mathfrak{g})$ where $\Orb^L \in \mathcal{N}_\mathfrak{l}/L$ is birationally rigid.

We state a remarkable result on birational sheets obtained by Losev, see \cite[Remark 4.11]{Losev}.
\begin{prop} \label{prop_losevconj}
If $\Orb_1$ and $\Orb_2$ are two orbits of $\mathfrak{g}$ lying in the same birational sheet, then their $G$-module structure is isomorphic.
\end{prop}

In addition, Losev conjectured that the viceversa is also true, giving hope for an intrinsic characterization of birational sheets of the Lie algebra.

\begin{cj} \label{conj}
If $\Orb_1$ and $\Orb_2$ are two orbits of $\mathfrak{g}$ with isomorphic $G$-module structure, then they lie in the same birational sheet.
\end{cj}

\subsection{Group case}
Before its introduction in the case of the adjoint action on the Lie algebra, induction was defined by Lusztig and Spaltenstein for unipotent conjugacy classes in a connected reductive algebraic group, see \cite{LS79}.
Consider a parabolic subgroup $P \leq G$ with Levi decomposition $P = LU_P$ and $\orb^L \in \mathcal{U}_L/L$.
 Then $P$ acts on $\overline{\orb^L} U_P $ via conjugacy and one can define the generalized Springer map:
\begin{equation}  \label{eq_GSM}
\gamma\colon G \times^P \overline{\orb^L} U_P \to G\cdot (\overline{\orb^L} U_P),\quad
g * x \mapsto gxg^{-1}.
\end{equation} 
The image of $\gamma$ is the closure of a single conjugacy class $\orb \in \mathcal{U}/G$, and $\ind_L^G (\orb^L) := \orb$ is the \emph{conjugacy class induced} from $(L, \orb^L)$.
When $\gamma$ is birational, we say that $\ind_L^G (\orb^L)$ is\textit{ birationally induced }from $(L, \orb^L)$.
If $\orb$ is a unipotent class in $G$ which cannot be induced (resp. birationally induced) from $(L, \orb^L)$ from any proper Levi subgroup $L$ of $G$ and $\orb^L \in \mathcal{U}_L/L$, we say that it is {\em rigid} (resp. {\em birationally rigid}).
All these notions are independent of the chosen parabolic subgroup $P$, see \cite[Lemma 3.5]{AF}.

\begin{rk}\label{rk_equiv}
Thanks to the bijective correspondences between parabolic subgroups, Levi subgroups, unipotent classes in $G$ and 
parabolic subalgebras, Levi subalgebras, nilpotent orbits in $\mathfrak{g}$, we have that $\gamma$ in \eqref{eq_gsm} is birational if and only if $\gamma$ in \eqref{eq_GSM}is so, see \cite[Remark 3.4]{AF}.
\end{rk}

\begin{defn} 
Consider a pseudo-Levi subgroup $M \leq G$, let $Z:=Z(M)$ and $z \in Z$.
We say that the connected component $Z^\circ z$ \textit{satisfies the regular property} (RP) for $M$ if
\begin{equation} \label{RP}
C_G(Z^\circ z)^\circ = M \tag{RP}.
\end{equation} 
\end{defn}

Observe that, for a pseudo-Levi subgroup $M \leq G$ and $z \in Z:=Z(M)$, we have that $Z^\circ z$ satisfies (\ref{RP}) for $M$ if and only if $Z^{reg} \cap Z^\circ z \neq \varnothing$ if and only if
$Z(M)= \langle Z^\circ, Z(G), z \rangle$ (see \cite[Remark 3.6]{CE1}) if and only if $M$ is a Levi subgroup of $C_G(z)^\circ$ (see \cite[Lemma 3.3]{AF}).

\begin{rk} \label{rk_centre}
Assume $G$ simple, let $M=L_\Theta$ for $\Theta \subset\widetilde{\Delta}$, let $s$ be such that $M=C(s)^\circ$ and set $Z:=Z(M)$.
Observe that $Z^\circ s$ satisfies \eqref{RP} for $M$.
Let $z \in Z$ such that $Z^\circ z$ satisfies \eqref{RP} for $M$, then, by \cite[Theorem 4.1]{CarnoBul} (see also \cite[Theorem 7]{SommersBC}), there is $w\in W$ such that $w(\Theta)=\Theta$ and $\dot{w}(Z^\circ z)\dot{w}^{-1}=Z^\circ \hat{z} s$  for a certain $\hat{z} \in Z(G)$.
Let $W_1=\{w\in W\mid wsw^{-1}s^{-1}\in Z^\circ Z(G)\}$, $W_2=\{w\in W\mid wsw^{-1}s^{-1}\in Z^\circ\}$.
The assignment  $w\mapsto wsw^{-1}s^{-1}Z^\circ$ defines a group homomorphism $ W_1\to \frac{Z^\circ Z(G)}{Z^\circ}$  with kernel $W_2$.
Then the number of different $G$-classes of pairs $(M, Z^\circ z)$ for a fixed $M$ with $Z^\circ z$ satisfying \eqref{RP} for $M$ is
\begin{equation} \label{eq_index}
d_M := \left[\frac{Z(G)}{Z(G)\cap Z^\circ}:W_1/W_2\right].
\end{equation}
\end{rk}

\begin{rk} \label{rk_itsalevi}
Let $L \leq G$ be a pseudo-Levi subgroup and let $Z := Z(L)$, then $L$ is a Levi subgroup if and only if $Z(L) = Z(G) Z(L)^\circ$ if and only if $Z(L)^\circ z$ satisfies \eqref{RP} for all $z \in Z(L)$.
\end{rk}

\begin{lem} \label{lem_cclevi}
Let $L \leq G$ be a Levi subgroup. Then  two  connected components of $Z := Z(L)$ are conjugate in $G$ if and only if they are equal.
\end{lem}
\begin{proof}
This is clear from Remark \ref{rk_itsalevi}.
\end{proof}

A \emph{decomposition datum} of $G$ consists of a triple $(M, Z(M)^\circ z, \orb^M)$ such that:
\begin{enumerate}
\item[(a)]  $M$ is a pseudo-Levi subgroup of $G$;
\item[(b)] $Z(M)^\circ z$ is a connected component  of $Z(M)$ satisfying (\ref{RP}) for $M$;
\item[(c)] $\orb^M$ is a unipotent conjugacy class of $M$.
\end{enumerate}
To any element $su \in G$ we can associate its decomposition datum $(C_G(s)^\circ, Z(C_G(s)^\circ)^\circ s, \orb^{C_G(s)^\circ}_u)$: any decomposition datum is of this form.

The set of all decomposition data of $G$ is denoted by $\mathscr{D}(G)$ and $G$ acts on this set by simultaneous conjugacy on the triples.

Two elements $g_1, g_2 \in G$ are said to be \emph{Jordan equivalent} if their decomposition data are conjugate in $G$.
The \emph{Jordan class} of $su$ is the set of all elements which are Jordan equivalent to $su$: it is denoted $J(su)$.

If $\tau = (C_G(s)^\circ, Z(C_G(s)^\circ)^\circ s, \orb^{C_G(s)^\circ}_u)$ is the decomposition datum of $su$, then
$$J(su) = J(\tau) = G \cdot ( (Z(C_G(s)^\circ)^\circ s)^{reg} \orb^{C_G(s)^\circ}_u).$$
The group $G$ is partitioned into its Jordan classes, which are finitely many locally closed irreducible subvarieties parametrized by the finite set $\mathscr{D}(G)/G$.
Jordan classes are unions of equidimensional conjugacy classes.
The closure of a Jordan class is a union of Jordan classes.

Sheets for the conjugacy action of $G$ on itself were studied in \cite{CE1}. They are parametrized by the $G$-equivalence classes of decomposition data $\tau=(M, Z(M)^\circ s, \orb^M) \in \mathscr{D}(G)$ with $\orb^{M} \in \mathcal{U}_M/M$ rigid: the sheet corresponding to $\tau$ is
$$
S(\tau):= \overline{J(\tau)}^{reg} =  \bigcup_{z \in Z(M)^\circ s} G \cdot \left( s \ind_{M}^{C_G(z)^\circ} \orb^{M}\right).
$$

In the remainder of the paper, unless differently specified, we work under the assumption $G'$ {\em simply-connected}: as a consequence, centralizers of semisimple elements are connected.

As in \cite[\S 5.1]{AF}, for $(M, Z(M)^\circ s, \orb^M) \in \mathscr{D}(G)$ we define the set:
$$\bir (Z(M)^\circ s, \orb^M) = \{z \in Z(M)^\circ s \mid \ind_M^{C_G(z)} \orb^M \mbox{ is birationally induced}\}.$$
This is an open subset of $Z(M)^\circ s$, independent of the parabolic group chosen for induction (\cite[Remark 5.2, Proposition 5.1]{AF}):  it contains  $(Z(M)^\circ s)^{reg}$, since 
$\orb^M=\ind_M^M \orb^M$ is birationally induced from $(M, \orb^M)$.
For $\tau = (M, Z(M)^\circ s, \orb^M) \in \mathscr{D}(G)$, the \emph{birational closure} of $J(\tau)$ is $$\overline{J(\tau)}^{bir} := \bigcup_{z \in \bir(Z(M)^\circ s, \orb^M)} G \cdot \left(z \ind_M^{C_G(z)} \orb^M \right).$$
Then $J(\tau) \subseteq \overline{J(\tau)}^{bir} \subseteq \overline{J(\tau)}^{reg}$: 
in particular, being $\overline{J(\tau)}^{reg}$ irreducible, it is contained in a sheet, hence so is $\overline{J(\tau)}^{bir}$.
In fact $\overline{J(\tau)}^{bir}$ is an irreducible locally closed subvariety of $G$ and a union of Jordan classes (\cite[Proposition 5.2, Corollary 5.3]{AF}).

\begin{defn}
We define the set
$$\mathscr{BB}(G) := \{(M, Z(M)^\circ s, \orb^M) \in \mathscr{D}(G) \mid \orb^M \in \mathcal{U}_M/M \mbox{ birationally rigid}\}.$$
For $\tau \in \mathscr{BB}(G)$, we define the \emph{birational sheet} of $G$ corresponding to (the class of) $\tau$ as
$\overline{J(\tau)}^{bir}$.
\end{defn}

It follows from \cite[Theorem 5.1]{AF} that the birational sheets of $G$ form a partition of $G$.

\begin{rk} \label{rk_lonely}
For $G$ semisimple, a birational sheet coincides with a single conjugacy class if and only if it is $\orb_{su}$ with $s$ isolated 
and $\orb^{C_G(s)}_u$ a birationally rigid unipotent class of $C_G(s)$.
\end{rk}

\subsection{Criteria for birational induction}\label{criteria}
We recollect some results from \cite[Lemmas 3.2, 3.6]{AF}: they will be used to classify birational sheets containing spherical conjugacy classes.

\begin{lem} \label{lem_centrinpar}
Let $P \leq G$ be a parabolic subgroup with Levi decomposition $P=LU$, let $\orb^L \in \mathcal{U}_L/L$, let $\orb = \ind_L^G \orb^L$ and let $\gamma$ be as in \eqref{eq_GSM}.
The following are equivalent:
\begin{enumerate}
\item[(i)] $\gamma$ is birational;
\item[(ii)] for all $x \in  \orb \cap \overline{\orb^L} U$, we have $C_G(x) = C_P(x)$;
\item[(iii)] there exists $x \in  \orb \cap \overline{\orb^L} U$ such that $C_G(x) = C_P(x)$.\hfill $\square$
\end{enumerate} 
\end{lem}

\begin{lem} \label{lem_compgp}
Let $\phi\colon \mathcal{N} \to \mathcal{U}$ denote a Springer's isomorphism and let $\overline{G}$ be the adjoint group in the same isogeny class of $G$.
Let $\nu \in \mathcal{N}$.
Suppose that $C_{\overline{G}}(\nu)$ is connected.
 If $\orb_{\phi(\nu)} = \ind_L^G \orb^L$ for a Levi subgroup $L \leq G$ and $\orb^L \in \mathcal{U}_L/L$, then $\orb_{\phi(\nu)}$ is birationally induced from $(L, \orb^L)$.\hfill$\square$
\end{lem}

\begin{rk} \label{rk_typea}
Let $G = {\rm SL}_{n}(\C)$,
 then the condition in Lemma  \ref{lem_compgp} is always fulfilled,
hence a unipotent class in $G$ (resp. a nilpotent orbit in $\lieg$) is rigid if and only if it is birationally rigid if and only if it is $\{1\}$ (resp. $\{0\})$, see \cite[Example 3.4]{AF}. Moreover,
 sheets coincide with sheets in $\mathfrak{g}$ and in $G$, see \cite[Corollary 5.4]{AF}.
\end{rk} 

\subsection{Birationally rigid unipotent classes} \label{sss_bruc}
In this section we assume $G$ simple and we recollect the complete list of birationally rigid conjugacy classes in $\mathcal{U}$ (equivalently of birationally rigid adjoint orbits in $\mathcal{N}$).

Namikawa gave in \cite{Nami2009} a criterion to test when a nilpotent orbit is birationally rigid for simple classical Lie algebras.
If $\mathfrak{g}$ is of type $\mathsf{A}$, then the only birationally rigid orbit is the only rigid orbit, i.e. the null orbit.
Now let $\mathfrak{g}$ be of type $\mathsf{B, C, D}$. Let $\mathbf{d} = [d_1, \dots, d_r]$ denote the partition corresponding to a nilpotent orbit $\Orb$. Then $\Orb$ is birationally rigid in $\mathfrak{g}$ if and only if $\mathbf{d}$ has \textit{full members}, i.e. $1 = d_r$ and $d_i - d_{i+1} \leq 1$ for all $i = 1, \dots, r-1$, with the exception of the case $\mathbf{d}=[2^{n-1},1^2]$ in $\mathsf{D}_n$ for $n = 2m+1, m \geq 1$, which is birationally induced as a Richardson orbit.

Fu worked out the exceptional types in \cite{Fu2010}: birationally rigid orbits coincide with rigid ones, except in  
type $\mathsf{E}_7$, where also $A_2 + A_1$ and $A_4 + A_1$ are birationally rigid,
and in type $\mathsf{E}_8$, where also $A_4 + A_1$ and $A_4 + 2A_1$ are birationally rigid.

For a complete list of rigid nilpotent orbits in the exceptional cases, see \cite[Appendix 5.7]{McGovern2002}. It follows that every spherical nilpotent orbit is (birationally) rigid, apart from $2A_1$  in type $\mathsf{E}_6$ and $(3A_1)''$ in type $\mathsf{E}_7$.

\begin{rk} 
Recall from  \cite[Lemma 3.9]{Borho} that all nilpotent orbits $\Orb$ in $\mathfrak{g}$ simple are characteristic, except for:
\begin{enumerate}[label=(\arabic*)]
\item $\mathfrak{g}$ of type ${\sf D}_4$: Aut$(\mathfrak{g})$ acts transitively on  $\{\Orb_{[4^2]}, \Orb_{[4^2]}', \Orb_{[5,1^3]}\}$ and on  $\{\Orb_{[2^4]}, \Orb_{[2^4]}', \Orb_{[3,1^5]}\}$.
\item $\mathfrak{g}$ of type ${\sf D}_{2m}, m \geq 3$: 
the graph automorphism permutes
$\Orb_\mathbf{d}$ and $\Orb'_\mathbf{d}$ for every very even partition $\mathbf{d} \vdash 4m$.
\end{enumerate}
It follows that all  birationally rigid nilpotent orbits in simple Lie algebras are characteristic, analogously for all birationally rigid unipotent classes in simple algebraic groups.
\end{rk}

\subsection{Birational sheets and translation by central elements}
Let $\tau:=(M,Z(M)^\circ s,\orb^M) \in \mathscr{D}(G)$. For each $z\in Z(G)$, let $\tau_z := (M, Z(M)^\circ z s,\orb^M)$. Then we have 
$\overline{J(\tau_z)}^{bir} = z \overline{J(\tau)}^{bir}$, so that the union of all  $\overline{J(\tau_z)}^{bir}$ as $z$ varies in $Z(G)$ is
\begin{equation}\label{eq_centre}
Z(G)\overline{J(\tau)}^{bir} := \bigcup_{z \in Z(G)} z \overline{J(\tau)}^{bir}.
\end{equation}
We shall be interested in $Z(G)\overline{J(\tau)}^{bir}$ for $\tau\in \mathscr{BB}(G)$: to describe it, it is enough to describe $\overline{J(\tau)}^{bir}$ and to count the number of birational sheets in $Z(G)\overline{J(\tau)}^{bir}$.

\begin{rk}\label{rk_characteristic}
 For $G$ simple, let $\tau := (M, Z(M)^\circ z, \orb^M) \in \mathscr{D}(G)$ and set $Z := Z(M)$.
We have seen that the number of different $G$-classes of pairs $(M, Z^\circ z)$ for a fixed $M$, with $Z^\circ z$ satisfying \eqref{RP} for $M$ equals the index $d_M = \left[ \frac{Z(G)}{Z(G)\cap Z^\circ} : W_1/W_2 \right]$, defined in Remark \ref{rk_centre}.
 If $\orb^M$ is characteristic in $M$,
 the number of different $G$-classes of triples $(M, Z(M)^\circ z, \orb^M)$ for fixed $M$ and $\orb^M$, with $Z^\circ z$ satisfying \eqref{RP} for $M$ is again the index $d_M$.
\end{rk}

\section{The ring of regular functions as an invariant}
In this section we analyse the behaviour of the ring of regular functions on adjoint orbits (resp. on conjugacy classes) belonging to the same Jordan class.

It is proven in \cite[\S 3.7]{broer_lectures} that all adjoint orbits in the same Jordan class are isomorphic as $G$-homogeneous spaces: in particular their ringsof regular functions are isomorphic as $G$-modules.
 We address the similar problem in the group case; we start with an easy observation.

\begin{rk} \label{rk_uptocentre} 
 Let $su \in G$ and $z \in Z(G)$.
Then $\orb_{su}$ and  $\orb_{zsu}=z\orb_{su}$ are isomorphic $G$-homogeneous spaces: $C_G(su) = C_G(zsu)$ and $\C [\orb_{su}] \simeq_G \C[\orb_{zsu}]$.
\end{rk}

\begin{prop} \label{prop_iso_jc_scgp}
Let $G$  be connected reductive with $G'$ simply-connected and let $J$ be a Jordan class in $G$.
For any pair of classes $\orb_1, \orb_2 \subset Z(G)J$ we have  $\C [\orb_1] \simeq_G \C[\orb_1]$.
\end{prop}
\begin{proof}
Suppose $J = J(\tau)$ for $\tau = (M, Z(M)^\circ s, \orb^M_u) \in \mathscr{D}(G)$.
We have $J(\tau) = G \cdot ((Z(M)^\circ s)^{reg}u)$ 
 and any conjugacy class $\orb$ in $J(\tau)$ is of the form $\orb=\orb_{s_1u}$ with $s_1\in
(Z(M)^\circ s)^{reg}$. Then $C_G(s_1 u) = C_G(s_1) \cap C_G(u) = M \cap C_G(u) = C_M(u)$, so that all conjugacy classes in $J$ are isomorphic as $G$-homogeneous spaces. The statement follows from Remark \ref{rk_uptocentre}.
\end{proof}

Notice that Proposition \ref{prop_iso_jc_scgp} does not hold in general for $G$ not simply-connected, as illustrated by the following example.

\begin{eg} \label{ss_psl2}
Consider $G = \mathrm{SL}_2(\C)$, let $\overline{G} = \mathrm{PSL}_2(\C)$ and let $\pi \colon G \to \overline{G}$, $\pi(g) = \bar{g}$ be the isogeny.
Let us consider the torus $T \leq G$ given by diagonal elements: $T = \{ t_k = \mathrm{diag}[k, k^{-1}] \mid k \in \C^\times\}$.
We have the following situation for the centralizer of a regular element $\overline{t}_k, k \neq \pm1$:
$$C_{\overline{G}}(\bar{t}_k)= \begin{cases}
\overline{H} := N_{\overline{G}}(\overline{T}) &\mbox{ if } k = \pm i;\\ 
\overline{T} = \overline{H}^\circ & \mbox{ if } k \notin \{\pm i,\pm 1\}\end{cases}$$

Observe that $\overline{G} \cdot( \overline{T}^{reg})$ is the Jordan class in $\overline{G}$ consisting of regular semisimple elements. Along $\overline{G} \cdot( \overline{T}^{reg})$ neither the $\overline{G}$-module structure nor the $\overline{G}$-homogeneous space structure of conjugacy classes is preserved:
\begin{table}[H]
\centering
{\renewcommand{\arraystretch}{1.5}
\begin{tabular}{|c|c|c|}
\hline
$k$  & $\C [ \orb^{\overline{G}}_{\bar{t}_k} ]$  & $\Lambda ( \orb^{\overline{G}}_{\bar{t}_k} )$ \\ \hline
$k = \pm i$ & 
$\C[\overline{G}/\overline{H}]$
 &
 $4 n \omega$        \\ \hline
$k \in \C^\times \setminus \{\pm 1, \pm i\}$ & 
$\C[\overline{G}/\overline{T}]$
 &
 $ 2n \omega$        \\ \hline
   \end{tabular}}
    \caption{Regular semisimple spherical classes in $\mathrm{PSL}_2(\C)$.}
  \label{tab_psl2}
\end{table}
\end{eg}

In the following, we show that the (birational) sheets in $\mathfrak{sl}_n(\C)$ are the unions of adjoint orbits whose rings of regular functions are isomorphic as $G$-modules, therefore proving Losev's conjecture in type $\mathsf{A}$.
We are indebted to Eric Sommers for suggesting the use of small modules in the proof of the following statement.

\begin{thm} \label{thm_sln}
Let $G$ be simple, with $\Lie(G) = \mathfrak{sl}_n(\C)$.
Then adjoint orbits $\Orb_1$ and $\Orb_2$ are in the same (birational) sheet of $\mathfrak{sl}_n(\C)$ if and only if $\C[\Orb_1]\simeq_G \C[\Orb_2]$ if and only if $\Lambda(\Orb_1) = \Lambda(\Orb_2)$.
\end{thm}
\begin{proof}
In one direction, we can assume $G$ adjoint. By \cite{Dix},
sheets of $\mathfrak{sl}_n(\C)$ are disjoint and are parametrized by the unique nilpotent orbit which they contain.
The assumptions of \cite[Theorem 6.3]{BK79} hold, since $\overline{\Orb}$ is normal for each $\Orb \in \mathcal{N}/G$ and $C_G(\nu)$ is connected for all $\nu \in \mathcal{N}$.
As a consequence, the $G$-module structure of the rings of regular functions on adjoint orbits is preserved along sheets.
To prove the converse implication, it is enough to show that the rings of regular functions on two distinct nilpotent orbits of $\mathfrak{sl}_n(\C)$ have non-isomorphic $G$-module structure. 
Let $\C$ denote the trivial representation of a group.
Let $\lambda \in \poids^+$ and let $V(\lambda)_0 = V(\lambda)^T$ be the zero weight subspace in $V(\lambda)$. 
Then $V(\lambda)_0$ is a $W$-module, in general reducible.
Let $L$ be a Levi subgroup of $G$ with Weyl group $W_L$.
By \cite[Proof of Corollary 1]{Broer}, if $V(\lambda)$ is small (i.e. if twice a root never occurs as a weight of $\lambda$), we have $V(\lambda)^L = (V(\lambda)_0)^{W_L} $ and, by Frobenius reciprocity, $\dim (V(\lambda)_0)^{W_L} = [\ind_{W_L}^W (\C):V(\lambda)_0]$.
In the case of $\mathfrak{sl}_n(\C)$, for every irreducible $S_n$-module $M$ there exists $\lambda$ in $\poids^+\cap \Z \Phi$ such that $V(\lambda)$ is small and $V(\lambda)_0 \simeq_{S_n} M$, see \cite[example p. 389]{Broer}.
Conjugacy classes of Levi subgroups of $\SL_n(\C)$ are indexed by partitions $\mathbf{d}=[d_1,\dots,d_{k}]$ of $n$ with $d_k>0$: the induced Richardson nilpotent class is $\Orb_{\mathbf{d}^t}$.
Let $L_{\mathbf{d}}$ be the standard Levi subgroup, with Weyl group $S_{\mathbf{d}}=S_{d_1}\times \cdots\times S_{d_k}$, corresponding to the partition $\mathbf{d}=[d_1,\dots,d_{k}]$.
We  know that $[\C [ \Orb_{\mathbf{d}^t} ] : V]
 = \dim V^{L_{\mathbf{d}}}$ for every simple ${\rm SL}_n(\C)$-module $V$.
We put $U_{\mathbf{d}}=\ind_{S_{\mathbf{d}}}^{S_n} (\C)$ and denote by $V_{\mathbf{d}}$ the simple $S_n$-module (Specht module) corresponding to ${\mathbf{d}}$. Then
\begin{equation} \label{Kotska}
U_{\mathbf{d}}=V_{\mathbf{d}}\oplus\bigoplus_{{\mathbf{f}}>{\mathbf{d}}}K_{\mathbf{f \, d}}V_{\mathbf{f}}
\end{equation} 
where the coefficients $K_{{\mathbf{f \, d}}}$ are the Kostka numbers and $<$ is the lexicographic total order on partitions of $n$.
Let $\mathbf{d}$, $\mathbf{f}$ be different partitions of $n$: we may assume $\mathbf{d} > \mathbf{f}$.
By the previous discussion, there exists a small simple ${\rm SL}_n(\C)$-module $V(\lambda)$ such that $V_{\mathbf{f}}\simeq_{S_n}V(\lambda)_0$. Then
$[\C[\Orb_{\mathbf{f}}]:V(\lambda)] = 1 \neq 0 = [\C[\Orb_{\mathbf{d}}]:V(\lambda)]$ and this allows to conclude. 
\end{proof}

Thanks to the fact that pseudo-Levi subgroups of ${\rm SL}_n(\C)$ are Levi subgroups, we get the following group analogue of Theorem \ref{thm_sln}.
\begin{thm} \label{thm_SLn}
Let $G = \mathrm{SL}_n(\C)$, $\orb_1$ and $\orb_2$ be conjugacy classes of $G$ and
let $S_1$ (resp. $S_2$) be the (birational) sheet containing $\orb_1$ (resp. $\orb_2$).
Then $\C[\orb_1] \simeq_G \C[\orb_2]$ if and only if $\Lambda(\orb_1) = \Lambda(\orb_2)$ if and only if $S_2= zS_1$ for some $z \in Z(G)$.
\end{thm}
\begin{proof}
Recall that (birational) sheets in $G$ are disjoint and  are parameterised by $G$-classes of pairs $(L, Z(L)^\circ z)$, with $L$ a Levi sugbroup of $G$ and certain $z$ in $Z(G)$, see  \cite[Corollary 5.4]{AF}.
For every $x\in G$ there exists ${\xi} \in \lieg$ such that $C_G(\xi)=C_G(x)$: if the sheet of $G$ containing $\orb_x$ corresponds to the $G$-class of $(L, Z(L)^\circ z)$, then the sheet of $\lieg$ containing $\xi$ corresponds to the $G$-class of $\Lie(L)$.
Let  $x_i \in \orb_i$ and ${\xi_i} \in \lieg$ such that $C_G(\xi_i)=C_G(x_i)$, and let $(L_i, Z(L_i)^\circ z_i)$ correspond to $S_i$ for $i=1, 2$.
Then $\C[\orb_i]=\C[\Orb_{\xi_i}]$ for $i=1, 2$.
Therefore $\C[\orb_1] \simeq_G \C[\orb_2]$ if and only if and only if $L_1$ and $L_2$ are $G$-conjugate by Theorem \ref{thm_sln}, hence if and only if  $S_2= zS_1$ for some $z \in Z(G)$.
\end{proof}

\section{Spherical birational sheets}\label{ZB}
This section is dedicated to the main result of the paper: the classification of spherical conjugacy classes grouped in birational sheets and the verification of the analogues of Proposition \ref{prop_losevconj} and Conjecture \ref{conj} in the case of $G$ connected reductive with $G'$ simply-connected.

The property of being spherical is preserved along sheets, as proven in \cite[Proposition 1]{Arzhantsev_1997}. A \emph{spherical sheet} is a sheet consisting of spherical orbits, as in \cite{CarnoBul}.
Since every birational sheet is irreducible, it is contained in a sheet, and the following definition is well-posed.

\begin{defn} Let $\tau \in \mathscr{BB}(G)$.
We say that the birational sheet
$\overline{J(\tau)}^{bir}$ is \emph{spherical} if one of the following equivalent properties is satisfied:
\begin{enumerate}
\item[(i)] all conjugacy classes $\orb \subset \overline{J(\tau)}^{bir}$ are spherical;
\item[(ii)] there exists a spherical conjugacy class $\orb \subset \overline{J(\tau)}^{bir}$;
\item[(iii)] $\overline{J(\tau)}^{bir}$ is contained in a spherical sheet.
\end{enumerate}
\end{defn}

As recalled in the Introduction, if $\orb$ (resp. $S$) is a conjugacy class (resp. a sheet) of $G$ and we denote by $w_{\orb}$ (resp. $w_S$) the unique element of $W$ such that $\orb\cap B w_{\orb} B$ is dense in $\orb$ (resp. $S\cap Bw_S B$ is dense in $S$), then if $S$ is spherical,
 for every conjugacy class $\orb$ lying in $S$ we have $w_{\orb}=w_S$.
 For a birational sheet $\overline{J(\tau)}^{bir}$ we may define $w_\tau$ as the unique element of $W$ such that $\overline{J(\tau)}^{bir}\cap Bw_\tau B$ is dense in $\overline{J(\tau)}^{bir}$.
  It follows that for a spherical birational sheet $\overline{J(\tau)}^{bir}$, we have $w_{\tau}=w_{\orb}$ for every conjugacy class $\orb \subset \overline{J(\tau)}^{bir}$ and $w_{\tau}=w_S$ for every sheet $S$ containing $\overline{J(\tau)}^{bir}$.
 
We state our main result. 
\begin{thm} \label{thm_principal}
Let $G$ be a  complex connected reductive algebraic group with $G'$ simply-connected.
Then the spherical birational sheets form a partition of $G_{sph}$.
Let $\orb_1$ and $\orb_2$ be spherical conjugacy classes in $G$.
Let $\overline{J(\tau_1)}^{bir}$ (resp. $\overline{J(\tau_2)}^{bir}$) be the birational sheet containing $\orb_1$ (resp. $\orb_2$).
Then $\C[\orb_1]$ is isomorphic to $\C[\orb_2]$ as a $G$-module if and only if $\overline{J(\tau_2)}^{bir}= z\overline{J(\tau_1)}^{bir}$ for some $z \in Z(G)$.
\end{thm}

Since the birational sheets form a partition of $G$, the spherical birational sheets form a partition of $G_{sph}$. 
The remainder of this section is devoted to the proof of Theorem \ref{thm_principal}: it is enough to assume $G$ simple.

If $z \in Z(G)$, then $\orb_z = \{z\}$, $w_{z}= 1$ and $\C[\orb_z] = \C$: then $\{z\}$ is the unique sheet and the unique birational sheet containing $z$.
We shall therefore only deal with non-central spherical conjugacy classes.

As recalled in the introduction, for a spherical conjugacy class $\orb$ the $G$-module structure of $\C[\orb]$ is completely determined by the weight lattice $\Lambda(\orb)$.

For $k = 1, \dots, n$, we put
\begin{align}
 \sigma_k &:= \exp \left(\dfrac{2\pi i}{c_k}\check{\omega}_k\right); \label{eq_sigma}\\
 \Theta_k &:= \Delta \setminus \{\alpha_k \} & L_k &:= L_{\Theta_k};  \label{eq_levi} \\
 \widetilde{\Theta}_k &:= \widetilde{\Delta} \setminus \{\alpha_k \} &M_k &:= L_{\widetilde{\Theta}_k} = C_G(\sigma_k). \label{eq_pseudo}
\end{align} 
We start from the list of proper spherical pseudo-Levi subgroups $H$ of $G$ arising as centralizers of semisimple elements.
We have the following possibilities:
\begin{enumerate}
\item[(i)] $H$ is a Levi subgroup of the form $L_i$ as in \eqref{eq_levi} for some  $i \in \{1,\dots,n\}$, with either $c_i=1$, so that $H$ is a maximal connected reductive subgroup of $G$, or $c_i=2$, so that  $H<H_1<G$, where $H_1$ is the spherical pseudo-Levi subgroup of $G$ corresponding to $\widetilde{\Theta}_i$ as in \eqref{eq_pseudo};
\item[(ii)] $H$ is not a Levi subgroup: then $H$ is of the the form $M_i$ as in \eqref{eq_pseudo}, with $c_i \in \{2, 3\}$, so that 
$H$ is a maximal connected reductive subgroup of $G$.
\end{enumerate}

For each such $H=C_G(s)$ with $s$ semisimple we consider certain birationally rigid unipotent conjugacy classes $\orb^H$ in $H$ and describe the birational sheet corresponding to $\tau=(H, Z^\circ s,\orb^H)$ by checking whether each class in $\overline{J(\tau)}^{reg} \setminus J(\tau)$ is birationally induced, using Lemmas \ref{lem_centrinpar} and \ref{lem_compgp}.
Eventually we are left, up to central elements, with spherical birationally rigid unipotent conjugacy classes in $G$.

For each type we collect the results in a table.
In the first column there is a certain $\tau=(M, Z(M)^\circ s, \orb_u^M) \in \mathscr{BB}(G)$ with $M=C_G(s)$ and $\orb_{su}$ spherical in $G$.
In the second column we describe $\overline{J(\tau)}^{bir}$. 
From the tables in \cite{Costantini2010} we verify that the weight monoid is constant on the orbits in $\overline{J(\tau)}^{bir}$ and we describe it in the third column.
In the cases when $Z(G)$ is non-trivial, we list also a fourth column indicating the number of (disjoint) birational sheets
in $Z(G)\overline{J(\tau)}^{bir}$.
This is produced by applying Remark \ref{rk_characteristic} in all cases, except for one case where $G$ is of type ${\sf C}_{2p}$ and $M$ is of type ${\sf C}_{p}{\sf C}_{p}$, see Remark \ref{rk_notchar}).

The fact that $\Lambda(\orb)$ is independent of the orbit $\orb$ in $\overline{J(\tau)}^{bir}$ (and hence in $Z(G)\overline{J(\tau)}^{bir}$) proves the group analogue of
Proposition \ref{prop_losevconj} (for spherical conjugacy classes in $G$ simple simply-connected). To prove the validity of the group analogue of
Conjecture \ref{conj} one has to check that the entries in the third column are pairwise distinct. 

\bigskip

We shall freely use the notation from \cite{Costantini2010}.
For $K \leq G$, $K$ simple, we will consider the isogeny $\pi_K\colon K\to \overline K:=K_{ad}$ to the adjoint group, omitting subscripts when $K=G$.

\subsection{Type $\mathsf{A}_n$, $n \geq 1$}
Here $G = \mathrm{SL}_{n+1}(\C)$, for $n \geq 1$.

\begin{lem} \label{lem_sln_levi_str}
Let $\mathbf{d}= [d_1, \dots, d_r]  \vdash n+1$ and let $L_{\mathbf{d}}$ be the standard Levi subgroup of $G$ indexed by $\mathbf{d}$. Then  $Z(L_{\mathbf{d}})$ has exactly $\gcd\{d_i \mid d_i \in \mathbf{d}\}$ connected components, pairwise not conjugate in $G$.
 \end{lem}
\begin{proof}
 We have
$Z(L)  \simeq S := \{(z_1, \dots, z_r) \in (\C^\times)^r \mid z_1^{d_1} \cdots z_r^{d_r} = 1\}$.  If $d =\gcd\{d_i \mid i=1, \dots, r\}$, we have $Z(L)/Z(L)^\circ \simeq S/S^\circ \simeq  \Z / d \Z$.
The last assertion follows from Lemma \ref{lem_cclevi}.
\end{proof}

For completeness we prove Theorem  \ref{thm_principal} for type $\mathsf{A}_n$, but it is a consequence of Theorem \ref{thm_SLn}.

\begin{prop}
Theorem \ref{thm_principal} holds for $G$ of type $\mathsf{A}_n$, $n\geq 1$.	
\end{prop}
\begin{proof}
Recall that every sheet in $G$ is a birational sheet, see Remark \ref{rk_typea}.

If $n = 1$, every conjugacy class of $G$ is spherical and there are three (birational) sheets: $\{-1\}, \{1\}$ and $G^{reg}$. 

Let $n \geq 2$. 
Consider the Levi subgroups $L_i$, for all $i = 1, \dots, \left\lfloor \frac{n + 1}{2} \right\rfloor$.
Then $L_i' \simeq  {\rm SL}_{n+1-i} \times{\rm SL}_{i} $, the centre $Z(L_i)$ is one-dimensional and consists of $d = \gcd(n+1-i, i)=\gcd(i,n+1)$ distinct connected components which are not conjugate in $G$.
Let $\mathbf{d}_i = [n+1-i, i]$ and let $\tau_i :=  (L_i, Z(L_i)^\circ, \{1\})$,
then $Z(L_i)^\circ = \exp( \C \check{\omega}_i)$,  $Z(G) \cap Z(L_i)^\circ$ has order $\frac{n+1}d$ and
\begin{align*}
\overline{J(\tau_i)}^{bir} = \overline{J(\tau_i)}^{reg}&= \bigcup_{z \in Z(L_i)^\circ} G \cdot ( z\ind_{L_i}^{C_G(z)} \{1\}) = \bigcup_{\zeta \in \C \setminus 2 \pi i \Z} \orb_{\exp(\zeta \check{\omega}_i)} \sqcup \bigsqcup_{z \in Z(G) \cap Z(L_i)^\circ } z \orb_{\mathbf{d}_i^t},
\end{align*}  by \cite[Theorem 7.2.3]{CollMcGov}. Moreover the unipotent class $\orb_{\mathbf{d}_i^t}$ is the class denoted by $X_{i}$ in \cite[\S 4.1]{Costantini2010}.
From \cite[Table 1, 2]{Costantini2010}
the weight monoid is preserved along the classes in $Z(G)S$ for any spherical (birational) sheet $S$. The entries in the third column of Table \ref{tab_an} are pairwise distinct.
\end{proof}

\begin{table}[H] 
\centering
{\renewcommand{\arraystretch}{1.5}
\begin{tabular}{|c|c|c|c|}
\hline
$\tau$                       & $\overline{J(\tau)}^{bir}$   &               $\Lambda(\orb)$ & $d$ \\ \hline
\makecell{$(L_\ell, Z(L_\ell)^\circ, \{1\})$\\ $\ell=1\ldots, m-1$}      & 
\makecell{$\bigcup\limits_{\zeta \in \C \setminus 2 \pi i \Z} 
\orb_{\exp(\zeta \check{\omega}_\ell)} \sqcup$ \\
$ \sqcup (Z(G) \cap Z(L_i)^\circ)X_\ell$}
  & $\sum\limits_{k=1}^\ell n_k (\omega_k+\omega_{n-k+1})$&$\gcd(\ell,n+1)$ \\ \hline
\makecell{$ (L_m, Z(L_m), \{1\})$\\ $n=2m$} &
$\bigcup\limits_{\zeta \in \C \setminus 2 \pi i \Z} 
\orb_{\exp(\zeta \check{\omega}_m)}
 \sqcup Z(G)X_m$
& $\sum\limits_{k=1}^m n_k (\omega_k+\omega_{n-k+1})$ &$1$\\ \hline
\makecell{$(L_m, Z(L_m)^\circ, \{1\})$\\$n+1=2m$} &
$\bigcup\limits_{\zeta \in \C \setminus 2 \pi i \Z} 
\orb_{\exp(\zeta \check{\omega}_m)}
 \sqcup \pm X_m$
& $\sum\limits_{k=1}^{m-1} n_k (\omega_k+\omega_{n-k+1})+2n_m\omega_m$ &$m$\\ \hline
  \end{tabular}}
    \caption{Type $\mathsf{A}_n, n \geq 1, m=\left\lfloor \frac{n+1}{2} \right\rfloor$.}
    \label{tab_an}
\end{table}

\subsection{Type $\mathsf{C}_n, n \geq 2$}

Here $G = \mathrm{Sp}_{2n}(\C)$, for $n \geq 2$. 
We have $\beta=2\alpha_1+2\alpha_2+\cdots+2\alpha_{n-1}+\alpha_n$, and $Z(G)=\langle\hat z\rangle$ with $\hat z =\prod_{i=1}^{\left\lfloor \frac {n+1}2 \right\rfloor}\alpha^\vee_{2i-1}(-1)$.
We set $p := \left\lfloor \frac{n}{2} \right\rfloor$.
For $k = 1, \dots, p$, 
the pseudo-Levi subgroup $M_k$ is of type ${\sf C}_{n-k}{\sf C}_{k}$.

\subsubsection{Type $\mathsf{C}_2$}
\begin{lem}
Let $G = {\rm Sp}_4(\C)$.
Let $S_i := \overline{J_G(\tau_i)}^{reg}$ with $\tau_i := (L_i, Z(L_i)^\circ, \{1\})$ for $i=1,2$.
Then $S_2 = \overline{J(\tau_2)}^{bir}$ is a birational sheet and $S_1 = \overline{J(\tau_1)}^{bir}\sqcup \orb_{[2^2]}$.
\end{lem}

\begin{proof}
Observe that $L_2$ is maximal and $Z(L_2)$
is connected.
Then 
$$S_2 = \bigcup_{z \in Z(L_2)} G \cdot ( z\ind_{L_2}^{C_G(z)} \{1\}) = G\cdot((Z(L_2))^{reg}) \sqcup \ind_{L_2}^{G} \{1\} \sqcup \hat{z} \ind_{L_2}^{G} \{1\}.$$
We have $\ind_{L_2}^G \{1\} = \orb_{[2^2]}$,
 and $u = x_{\beta_1}(1)x_{\beta_2}(1)\in \orb_{[2^2]}=X_2$ 
 satisfies $C_G(u) \leq P_{\Theta_2}$,
  so that $\orb_{[2^2]}$ is birationally induced from $(L_2, \{1\})$ by  Lemma \ref{lem_centrinpar} and $S_2$ is a birational sheet. 

For $S_1$, observe that $L_1<M_1<G$, where  $Z(L_1) = Z(L_1)^\circ \sqcup \hat{z} Z(L_1)^\circ$.
We have $$S_1 = \bigcup_{z \in Z(L_1)^\circ} 
G \cdot ( z \ind_{L_1}^{{C_G(z)}} \{1\}) = 
G \cdot ((Z(L_1)^\circ)^{reg}) \sqcup G \cdot (\sigma_1 \ind_{L_1}^{C_1} \{1\}) \sqcup \ind_{L_1}^G \{1\}.$$
Observe that $M_1$ is of type ${\sf A}_1{\sf A}_1$, so the class $\ind_{L_1}^{M_1} \{1\}= \orb^{M_1}_{x_\beta(1)}$ is birationally induced by Remark \ref{rk_typea}.
The subregular unipotent class $\orb_{[2^2]} = \ind_{L_1}^G \{1\}$ is not birationally induced from $(L_1, \{1\})$, so that 
$\overline{J(\tau_1)}^{bir} =\bigcup_{\zeta \in \C \setminus \pi i \Z}  \orb_{\exp(\zeta \check{\omega}_1)}  \sqcup  \orb_{\sigma_1 x_{\beta}(1)}$ and
$Z(G)\overline{J(\tau_1)}^{bir} =\overline{J(\tau_1)}^{bir} \sqcup \hat z \overline{J(\tau_1)}^{bir}
$.
\end{proof}

There is only one more spherical pseudo-Levi subgroup $M_1$ giving rise to the (birational) sheet $\orb_{\sigma_1}$. Note that $\sigma_1$ and $\hat z\sigma_1$ are conjugate, hence $Z(G)\orb_{\sigma_1}=\orb_{\sigma_1}$.
  
Up to central elements, there is only one more spherical conjugacy classes $X_1$ corresponding to the partition $[2,1^2]$: this is a birationally rigid unipotent conjugacy class in $G$.

\begin{table}[H]
\centering
{\renewcommand{\arraystretch}{1.5}
\begin{tabular}{|c|c|c|c|}
\hline
$\tau$                       & $\overline{J(\tau)}^{bir}$                 & $\Lambda(\orb)$ &$d$ \\ \hline
$ (L_2, Z(L_2), \{1\})$      & 
$\bigcup\limits_{\zeta \in \C \setminus 2 \pi i \Z} 
\orb_{\exp(\zeta \check{\omega}_2)}
 \sqcup X_2\sqcup \hat{z} X_2$
  & $2 n_1\omega_1 + 2 n_2\omega_2$&$1$ \\ \hline
$ (L_1, Z(L_1)^\circ, \{1\})$ & $\bigcup\limits_{\zeta \in \C \setminus \pi i \Z} \orb_{\exp(\zeta \check{\omega}_1)}  \sqcup  \orb_{\sigma_1 x_{\beta}(1)}$ &$2 n_1 \omega_1 + n_2 \omega_2$&$2$ \\ \hline
$ (M_1, \{\sigma_1\}, \{1\})$  & $\orb_{\sigma_1}$ & $n_2 \omega_2$   &$1$     \\ \hline
$(G, \{1\}, \orb_{[2,1^2]})$ & $X_1$ &  $2n_1\omega_1$ &  	$2$	\\ \hline
\end{tabular}}
\caption{Type $\mathsf{C}_2$.}\label{tab_c2}
\end{table}

\begin{prop}
Theorem \ref{thm_principal} holds for $G$ of type $\mathsf{C}_{2}$.
\end{prop}
\begin{proof}
From \cite[Table 3, 4, 5]{Costantini2010}
the weight monoid is preserved along the classes in each $Z(G)\overline{J(\tau)}^{bir}$. The entries in the third column of Table \ref{tab_c2} are pairwise distinct.
\end{proof}

\begin{rk}
The subregular unipotent class $\orb_{[2^2]}$ lies in both the sheets $S_1$ and $S_2$.
This agrees with what is stated in \cite[\S 6(c)]{BK79}: $\orb_{[2^2]}$ can be deformed in semisimple classes of both types $\orb_{\exp(\zeta \check{\omega}_1)}$ and $\orb_{\exp(\zeta \check{\omega}_2)}$, but in general the multiplicities of the weights can decrease. Indeed, for  $\zeta \in \C \setminus  \pi i \Z$, we have $\Lambda(\orb_{\exp(\zeta \check{\omega}_1)}) = \langle 2\omega_1, \omega_2 \rangle >  \langle 2\omega_1, 2\omega_2 \rangle = \Lambda(\orb_{[2^2]}) = \Lambda(\orb_{\exp(\zeta \check{\omega}_2)})$. 
\end{rk}

\begin{rk} 
The sheet $S_1$ is not a union of birational sheets.
\end{rk}

\subsubsection{Type $\mathsf{C}_n$, $n \geq 3$}
\begin{lem}
Let $n \geq 3$. Then:
\begin{enumerate}
\item[(i)] Let $\tau_1 = (L_1, Z(L_1)^\circ, \{1\})$; then the spherical sheet $S_1 := \overline{J(\tau_1)}^{reg}$ decomposes as the union of $\overline{J(\tau_1)}^{bir}$
and the unipotent birationally rigid class $\orb_{\mathbf{d}}$ with $\mathbf{d} = [2^2,1^{2(n-1)}]$;
similarly for the birational sheet $\hat{z} S_1$.
\item[(ii)] Let $\tau_n  = (L_n, Z(L_n), \{1\})$; then the spherical sheet $S_n := \overline{J(\tau_n)}^{reg}$ is a birational sheet containing the unipotent class $\orb_{\mathbf{f}}$, with $\mathbf{f} = [2^n]$.
\end{enumerate}
\end{lem}

\begin{proof}
(i) $L_1$ is of type ${\sf T}_1 {\sf C}_{n-1}$ and $L_1<M_1<G$
 and
$Z(L_1) = Z(L_1)^\circ \sqcup \hat{z} Z(L_1)^\circ$.
Then
$$S_1 = \bigcup_{z \in Z(L_1)^\circ} G \cdot (z \ind_{L_1}^{C_G(z)} \{1\}) =
 G \cdot ((Z(L_1)^\circ)^{reg}) \sqcup 
 G \cdot (\sigma_1 \ind_{L_1}^{M_1} \{1\}) \sqcup \ind_{L_1}^G \{1\}.$$
The class $\ind_{L_1}^{M_1} \{1\}$ is birationally induced, by Remark \ref{rk_typea}.
 The unipotent class $\orb_\mathbf{d} = \ind_{L_1}^G \{1\}$ is not birationally induced from $(L_1, \{1\})$, indeed it is birationally rigid by \S \ref{sss_bruc} and it coincides with a whole  birational sheet.
Hence
 $$
 \overline{J(\tau_1)}^{bir}=\bigcup\limits_{\zeta \in \C \setminus \pi i \Z} \orb_{\exp(\zeta \check{\omega}_1)}  \sqcup  \orb_{\sigma_1 x_{\beta_1}(1)}
 $$
 and
 $Z(G) \overline{J(\tau_1)}^{bir}= \overline{J(\tau_1)}^{bir}\sqcup \hat z \overline{J(\tau_1)}^{bir}$.
 
(ii) $L_n$ is maximal of type ${\sf T}_1\sf{\tilde A}_{n-1}$ and $Z(L_n)$ is connected, as $L_n=C_G(\exp \check \omega_n)$ and $\exp (2\pi i\check \omega_n)=\hat z$.
$$S_n = \bigcup_{z \in Z(L_n)} G \cdot (z \ind_{L_n}^{C_G(z)} \{1\}) = G \cdot (Z(L_n)^{reg}) \sqcup  \ind_{L_n}^{G} \{1\} \sqcup\hat{z} \ind_{L_n}^{G} \{1\}.$$
We have $\ind_{L_n}^{G} \{1\} = \orb_{\mathbf{f}}$, with 
$u = x_{\beta_1}(1)\cdots x_{\beta_n}(1)\in \orb_{\mathbf{f}}=X_n$ 
 satisfies $C_G(u) \leq P_{\Theta_n}$,
  so that $X_n$ is birationally induced and $S_n$ is a birational sheet
 $$
 \overline{J(\tau_n)}^{bir} =\overline{J(\tau_n)}^{reg}=S_n=\bigcup\limits_{\zeta \in \C \setminus 2 \pi i \Z} 
\orb_{\exp(\zeta \check{\omega}_n)}
 \sqcup X_n \sqcup \hat z X_n$$
 and
 $Z(G) \overline{J(\tau_n)}^{bir}= \overline{J(\tau_n)}^{bir}$.
 \end{proof}
 
We consider the remaining spherical pseudo-Levi subgroups.
\begin{enumerate}[label=(\roman*)]
\item For $\ell = 1, \dots, p$,
$M_\ell$ is maximal of type ${\sf C}_\ell {\sf C}_{n-\ell}$ and $Z(M_\ell)=\langle\sigma_\ell\rangle\times Z(G)$. 
Then, for $(M_\ell, \{\sigma_\ell\},\{1\})$ we get $\orb_{\sigma_\ell}$, a (birational) sheet consisting of an isolated class. 
We have $Z(G) \orb_{\sigma_\ell}= \orb_{\sigma_\ell} \sqcup \hat z \orb_{\sigma_\ell}$ except when $n=2p$, $\ell=p$, in which case $\sigma_\ell$ and $\hat z \sigma_\ell$ are $G$-conjugate and $Z(G) \orb_{\sigma_p}= \orb_{\sigma_p}$.
\item For $\ell = 2, \dots, p$, the pseudo-Levi  $M_\ell$ of type ${\sf C}_\ell {\sf C}_{n-\ell}$  admits the birationally rigid unipotent class $\orb_{x_{\beta_1}(1)}^{M_\ell}$
 of the form
 $ [2, 1^{2\ell-2}] \times \{1\}$.
Then $\orb_{\sigma_\ell x_{\beta_1}(1)}$ is a (birational) sheet consisting of an isolated class. 
\item  For $\ell = 1, \dots, p$, the pseudo-Levi $M_\ell$ of type $  {\sf C}_\ell {\sf C}_{n-\ell}$ has the birationally rigid unipotent class $\orb_{x_{\alpha_n}(1)}^{C_G(\sigma_\ell)}$ of the form
 ${\{1\}\times[2, 1^{2(n-\ell)-2}]}$.
Then $\orb_{\sigma_\ell x_{\alpha_n}(1)}$ is a (birational) sheet consisting of an isolated class.

In cases (ii) and (iii), we have
 $Z(G) \orb_{\sigma_\ell x_{\beta_1}(1)}= \orb_{\sigma_\ell x_{\beta_1}(1)} \sqcup \hat z \orb_{\sigma_\ell x_{\beta_1}(1)}$ and $Z(G) \orb_{\sigma_\ell x_{\alpha_n}(1)}= \orb_{\sigma_\ell x_{\alpha_n}(1)} \sqcup \hat z \orb_{\sigma_\ell x_{\alpha_n}(1)}$. 
  The only case which needs further explanation is when $n=2p$, $\ell=p$: then  $\sigma_p$ and $\hat z \sigma_p$ are $G$-conjugate, but
  $\sigma_p x_{\beta_1}$ and $\hat z \sigma_p x_{\beta_1}$ are not $G$-conjugate.
\end{enumerate}

\begin{rk}  \label{rk_notchar}
This is an example of $(M,Z^\circ s_1, \orb^M)$, $(M,Z^\circ s_2, \orb^M)$ in $\mathscr{BB}(G)$ with $(M,Z^\circ s_1)$, $(M,Z^\circ s_2)$ $G$-conjugate, but $(M,Z^\circ s_1, \orb^M)$, $(M,Z^\circ s_2, \orb^M)$ not $G$-conjugate: in this case the rigid orbit $\orb^M$ of $M$ is not characteristic.
\end{rk}

Up to central elements, the remaining spherical conjugacy classes in $G$ are $X_\ell$ corresponding to the partition $[2^\ell,1^{2n-2\ell}]$, for $\ell=1,\ldots,n-1$: these are all birationally rigid unipotent conjugacy classes in $G$, see \S \ref{sss_bruc}.

\begin{table}[H] 
\centering
{\renewcommand{\arraystretch}{1.5}
%\hskip-1.5cm
\begin{tabular}{|c|c|c|c|}
\hline
$\tau$                       &$\overline{J(\tau)}^{bir}$   &               $\Lambda(\orb)$ & $d$ \\ \hline
$ (L_n, Z(L_n), \{1\})$      & 
$\bigcup\limits_{\zeta \in \C \setminus  \pi i \Z} 
\orb_{\exp(\zeta \check{\omega}_n)}
 \sqcup X_n\sqcup \hat z X_n$
  & $\sum\limits_{i=1}^{n} 2 n_i\omega_i$&1 \\ \hline
$ (L_1, Z(L_1)^\circ, \{1\})$ &
$\bigcup\limits_{\zeta \in \C \setminus \pi i \Z} \orb_{\exp(\zeta \check{\omega}_1)}  \sqcup  \orb_{\sigma_1 x_{\beta_1}(1)}$
& $2 n_1\omega_1 +  n_2\omega_2$ &2\\ \hline
\makecell{$(M_\ell, \{\sigma_{\ell}\}, \{1\})$\\ $\ell = 1, \dots, p-1$} & $\orb_{\sigma_\ell}$
&  $ \sum\limits_{i=1}^{\ell} n_{2i}\omega_{2i}$&2 \\\hline
\makecell{$ (M_p, \{\sigma_p\}, \{1\})$\\ if $n=2p +1$}
& \multirow{2}{*}{ $\orb_{\sigma_p}$}
& \multirow{2}{*}{ $ \sum\limits_{i=1}^{p} n_{2i}\omega_{2i}$}&2 \\\cline{1-1}\cline{4-4}
\makecell{$ (M_p, \{\sigma_p\}, \{1\})$\\ if $n=2p$}
&  & &1\\ \hline
$ (M_p, \{\sigma_p\}, \orb^{M_p}_{\{1\}\times[2, 1^{2(n- p)-2}]})$ & 
$\orb_{\sigma_p x_{\alpha_n}(1)}$ &    
$\sum\limits_{i = 1}^n n_i \omega_i \mid \sum\limits_{i = 1}^{\left\lfloor \frac{n+1}{2} \right\rfloor} n_{2i-1} \in 2\N$ &2 \\ \hline
\makecell{$ (M_\ell, \{\sigma_\ell\}, \orb^{M_\ell}_{\{1\}\times[2, 1^{2(n-\ell)-2}]})$\\$\ell = 1, \dots, p-1$} & 
$\orb_{\sigma_\ell x_{\alpha_n}(1)}$ &  
$\sum\limits_{i = 1}^{2\ell+1} n_i \omega_i \mid \sum\limits_{i = 1}^{\ell +1} n_{2i-1} \in 2\N$  &2  \\ \hline
\makecell{$ (M_\ell, \{ \sigma_\ell\}, \orb^{M_\ell}_{[2, 1^{2\ell-2}] \times \{1\}})$\\ $\ell = 2, \dots, p$} & 
$\orb_{\sigma_\ell x_{\beta_1}(1)}$ &        
$\sum\limits_{i = 1}^{2\ell} n_i \omega_i \mid \sum\limits_{i = 1}^{\ell} n_{2i-1} \in 2\N$ &2 \\ \hline
\makecell{$(G, \{1\}, \orb_{[2^{\ell},1^{2n - 2\ell}]})$\\ $\ell = 1, \dots, n-1 $ }& $X_\ell$ & $ \sum\limits_{i=1}^{\ell} 2n_i\omega_{i}$ 
& 2  	\\ \hline
\end{tabular}}
\caption{Type $\mathsf{C}_n, n \geq 3, p = \left\lfloor \frac{n}{2}\right\rfloor$.} \label{tab_cn}
\end{table}

\begin{prop}
Theorem \ref{thm_principal} holds for $G$ of type $\mathsf{C}_n$, $n\geq 3$.	
\end{prop}
\begin{proof}
From \cite[Table 3, 4, 5]{Costantini2010}
the weight monoid is preserved along the classes in each $Z(G)\overline{J(\tau)}^{bir}$.
Assume $\Lambda(\orb_1)=\Lambda(\orb_2)$
 for $\orb_1 \subset Z(G)\overline{J(\tau_1)}^{bir}, \orb_2 \subset Z(G)\overline{J(\tau_2)}^{bir}$. 
This is possible if and only if $n=2p$ for $p \in \N, p \geq 1$, and 
$$
\tau_1=(C_p, \{\sigma_p\}, \orb^{C_p}_{\{1\}\times[2, 1^{n-2}]})=
(M_p, \{\sigma_p\}, \orb^{M_p}_{x_{\alpha_n(1)}})
$$
$$
\tau_2=(C_p, \{\sigma_p\}, \orb^{C_p}_{[2, 1^{n-2}]\times \{1\}})=
(M_p, \{\sigma_p\}, \orb^{M_p}_{x_{\beta_1(1)}})
$$
However in this case $\sigma_p$ and $\hat z\sigma_p$ are $G$-conjugate, and so are $\sigma_p x_{\alpha_n}(1)$ and $\hat z\sigma_p x_{\beta_1}(1)$.
 Therefore $\tau_1$ and $(M_p, \{\hat z
\sigma_p\}, \orb^{M_p}_{x_{\beta_1(1)}})$ are $G$-conjugate, i.e. $\overline{J(\tau_1)}^{bir}=\hat z \overline{J(\tau_2)}^{bir}$.
\end{proof}

\subsection{Type $\mathsf{B}_n, n \geq 3$}
We have $\beta=\alpha_1+2\alpha_2+\cdots+2\alpha_{n-1}+2\alpha_n$ and
$Z(G)=\langle\hat{z}\rangle$ with $\hat{z} = \alpha^\vee_n(-1)$.

The following result holds indepedently of the parity of $n$.

 \begin{lem}
Let $\tau_1 = (L_1, Z(L_1), \{1\})$. Then $S_1 := \overline{J(\tau_1)}^{reg} = \overline{J(\tau_1)}^{bir}$ is a spherical birational sheet of $G$ containing the unipotent class $\orb_{\mathbf{d}}$, with $\mathbf{d} = [3,1^{2n-2}]$.
\end{lem}
\begin{proof}
$L_1$ is maximal of type ${\sf T}_1 {\sf B}_{n-1}$  and $Z(L_1)$ is connected since $L_1=C_G(\exp \check \omega_1)$ and $\exp (2\pi i\check \omega_1)=\hat z$.
$$
S_1 = \bigcup\limits_{z \in Z(L_1)} G \cdot ( z\ind_{L_1}^{C_G(z)} \{1\})
= G \cdot (Z(L_1)^{reg}) \sqcup
  \orb_{\mathbf{d}} \sqcup \hat z
  \orb_{\mathbf{d}}.
  $$
 We have $\ind_{L_1}^{G} \{1\} = \orb_{\mathbf{d}}$, with $\mathbf{d} = [3,1^{2n-2}]$
 and
$u = x_{\gamma_1}(1)\in \orb_{\mathbf{d}}=Z_1$, where $\gamma_1=e_1$ is the highest short root of $G$, 
 satisfies $C_G(u) \leq P_{\Theta_1}$,
  so that $Z_1$ is birationally induced and $S_1$ is a birational sheet. Therefore
$$
{S}_1 = \bigcup_{\zeta \in \C \setminus 2\pi i \Z} \orb_{\exp(\zeta \check{\omega}_1)}\sqcup Z_{1}\sqcup \hat z Z_{1}
$$
and
 $Z(G) \overline{J(\tau_1)}^{bir}= \overline{J(\tau_1)}^{bir}$.
\end{proof}

\subsubsection{Type $\mathsf{B}_{2m+1}, m \geq 1$}
In this section we deal with cases $n = 2m+ 1$, $m \geq 1$.
 We set $u_k:=\prod_{i=1}^{k} x_{\beta_i} (1)$ for $k=1, \dots, m$.

 \begin{lem}
Let $\tau_n = (L_n, Z(L_n), \{1\})$ and let $S_n := \overline{J(\tau_n)}^{reg}$.
Then $S_n = \overline{J(\tau_n)}^{bir}$ is a spherical birational sheet in $G$.
\end{lem}
\begin{proof}
$L_n$ is of type ${\sf T}_1{\sf A}_{n-1}$ and $L_n<M_n<G$, where 
$Z(L_n)$ is connected since $L_n=C_G(\exp \check \omega_n)$ and $\sigma_n^2=\hat z$.
Then 
\begin{align*} 
S_n &= \bigcup_{z \in Z(L_n)} G \cdot ( z\ind_{L_n}^{C_G(z)} \{1\})=\\
&=G \cdot (Z(L_n)^{reg} )\cup G \cdot ( \sigma_n \ind_{L_n}^{M_n} \{1\}) \cup G \cdot ( \sigma_n^{-1} \ind_{L_n}^{M_n} \{1\}) \cup \ind_{L_n}^G \{1\}\cup \hat z\ind_{L_n}^G \{1\}.
 \end{align*}
We show that $\ind_{L_n}^{{M_n}} \{1\}$ is birationally induced from $(L_n, \{1\})$.
Observe that $M_n$ is of type ${\sf D}_n$ and
$\ind_{L_n}^{M_n} \{1\}$ is the unipotent class corresponding to $[2^{2n-1},1^2]$  in $\SO_{2n}(\C)$.
Let $K:=M_n$, 
for $u \in \ind_{L_n}^{K}$, we have  $C_{\overline K}(\pi_K(u))$ is connected by \cite[p. 399]{Carter2}, so the claim follows.
Also $\ind_{L_n}^G \{1\}$, the unipotent class corresponding to the partition $[3,2^{n-1}]$, denoted by $Z_{m+1}$ in \cite{Costantini2010}, is birationally induced from $(L_n, \{1\})$.
Indeed, for $u \in \ind_{L_n}^G \{1\}$,  the centralizer $C_{\overline G}(\pi(u))$  is connected by \cite[p. 399]{Carter2}, and we conclude.
Thus $S_n$ is a birational sheet in $G$. 
 We observe moreover that $G \cdot ( \sigma_n \ind_{L_n}^{M_n} \{1\})=
G \cdot ( \sigma^{-1}_n \ind_{L_n}^{M_n} \{1\}) = \orb_{\sigma_n u_m}$, 
as $w_0$ conjugates $\sigma_n$ to its inverse and $\orb^{M_n}_{u_m}$ is characteristic in $M_n$.
Therefore
$$
{S}_n = \overline{J(\tau_n)}^{bir} = \bigcup_{\zeta \in \C \setminus \pi i \Z} \orb^{G}_{\exp(\zeta \check{\omega}_n)}\sqcup  \orb_{\sigma_n u_m} \sqcup Z_{m+1}\sqcup \hat z Z_{m+1}
$$ 
and $Z(G)S_n=S_n$.
\end{proof}

We consider the remaining spherical pseudo-Levi subgroups:
\begin{enumerate}[label=(\roman*)]
\item For $\ell = 2, \dots, n$, the pseudo-Levi $M_\ell$ is maximal of type ${\sf D}_\ell {\sf B}_{n-\ell}$.
If $\ell$ is even we have $\sigma_\ell^2=1$ and 
$Z(M_\ell)=\langle\sigma_\ell\rangle\times Z(G)$;
if $\ell$ is odd we have $\sigma_\ell^2=\hat z$ and 
$Z(M_\ell)=\langle\sigma_\ell\rangle$.
In any case $\sigma_\ell$ and $\hat z\sigma_\ell$ are $G$-conjugate (via the reflection $s_{e_1}$).
Then $\orb_{\sigma_\ell}$
is a (birational) sheet consisting of an isolated class, and $Z(G)\orb_{\sigma_\ell}=\orb_{\sigma_\ell}$.
\item For $\ell = n$, the subgroup $M_n$ of type ${\sf D}_n$ admits the birationally rigid unipotent class  $\orb_{u_k}^{M_n}$,
 corresponding in ${\rm SO}_{2n}(\C)$ to the partition $[2^{2k}, 1^{2(n-2k)}]$, for $k = 1, \dots, m -1$.
Since $\sigma_n \sim_W \sigma_n^{-1}$ by $s_n$ and $\orb^{M_n}_{u_k}$ is characteristic, $Z(G) \orb_{\sigma_n u_k}=\orb_{\sigma_n u_k}$ is a (birational) sheet consisting of an isolated class.
\end{enumerate}

Up to central elements, the remaining spherical conjugacy classes in $G$ are $X_\ell$ corresponding to the partition $[2^{2 \ell}, 1^{2n + 1 - 4 \ell}]$, for $\ell=1,\ldots,m$ and $Z_\ell$ corresponding to the partition $[3,2^{2( \ell - 1)}, 1^{2n + 2 - 4 \ell}]$, for $\ell=2,\ldots,m$: these are all birationally rigid unipotent conjugacy classes in $G$.

\begin{table}[H]
\centering
{\renewcommand{\arraystretch}{1.5}
\begin{tabular}{|c|c|c|c|}
\hline
$\tau$                       & $\overline{J(\tau)}^{bir}$                 & $\Lambda(\orb)$ & $d$ \\ \hline
$ (L_1, Z(L_1), \{1\})$ & $\bigcup\limits_{\zeta \in \C \setminus 2 \pi i \Z} \orb_{\exp(\zeta \check{\omega}_1) } \sqcup  Z_1\sqcup  \hat zZ_1$ & $2n_1 \omega_1 +  n_2 \omega_2$  & $1$ \\ \hline
$ (L_n, Z(L_n), \{1\})$ & 
\makecell{$\bigcup\limits_{\zeta \in \C \setminus \pi i \Z} \orb_{\exp(\zeta \check{\omega}_n)} \sqcup $\\$\sqcup \orb_{\sigma_n u_m} \sqcup  Z(G)Z_{m+1}$ }&
$\sum\limits_{i=1}^{n-1} n_i\omega_i + 2n_n\omega_n $       & $1$ \\ \hline
\makecell{$(M_\ell, \{\sigma_\ell\}, \{1\})$\\$ \ell = 2, \dots, m$ }&
$\orb_{\sigma_\ell}$ &
$\sum\limits_{i = 1}^{2 \ell -1} 2 n_i\omega_i + n_{2\ell}\omega_{2\ell}$& $1$ \\ \hline
\makecell{$(M_\ell, \{\sigma_\ell\}, \{1\})$\\$ \ell = m+2, \dots, n$} &
$\orb_{\sigma_\ell}$ &
$\sum\limits_{i = 1}^{2(n - \ell)} 2n_i\omega_i + n_{2(n -\ell) + 1}\omega_{2(n -\ell) + 1}$& $1$ \\ \hline
$(M_{m+1}, \{\sigma_{m+1}\}, \{1\})$ &
$\orb_{\sigma_{m+1}}$ &
$\sum\limits_{i = 1}^n 2 n_i \omega_i$& $1$ \\ \hline
\makecell{$ (M_n, \{\sigma_n\}, \orb^{M_n}_{u_\ell})$\\ $\ell = 1, \dots, m-1$} &
$\orb_{\sigma_n u_\ell }$ &
$\sum\limits_{i = 1}^{2 \ell + 1} n_i\omega_i$& $1$ \\ \hline
\makecell{$(G,\{1\}, \orb_{[2^{2 \ell}, 1^{2n + 1 - 4 \ell}]})$ \\ $ \ell = 1, \dots, m$ }&
$X_\ell$ &
$\sum\limits_{i = 1}^{ \ell } n_{2i} \omega_{2i}$ & $2$\\ \hline
\makecell{$(G,\{1\}, \orb_{[3,2^{2( \ell - 1)}, 1^{2n + 2 - 4 \ell}]})$\\$ \ell = 2, \dots, m$}&
$Z_\ell$ &
$\sum\limits_{i = 1}^{ 2\ell } n_i \omega_{i} \mid \sum\limits_{i = 1}^{ \ell }  n_{2i-1} \in 2\N$& $2$ \\ \hline
  \end{tabular}}
  \caption{Type $\mathsf{B}_{n}$, $n=2m+1$, $m \geq 1$.} 
 \label{tab_bnodd}
\end{table}

\begin{prop}
Theorem \ref{thm_principal} holds for $G$ of type $\mathsf{B}_{2m+1}$, for $m \geq 1$.	
\end{prop}
\begin{proof}
From \cite[Table 13, 14, 15]{Costantini2010}
the weight monoid is preserved along the classes in each $Z(G)\overline{J(\tau)}^{bir}$. The entries in the third column of Table \ref{tab_bnodd} are pairwise distinct.
\end{proof}

\subsubsection{Type $\mathsf{B}_{2m}, m \geq 2$}
In this section we assume $n = 2m$, $m \geq 2$.
 We set $u_k:=\prod_{i=1}^{k} x_{\beta_i} (1)$ for $k=1, \dots, m$.
\begin{lem}
Let $\tau_n = (L_n, Z(L_n)^\circ, \{1\})$ and let $S_n := \overline{J(\tau_n)}^{reg}$.
Then $S_n = \overline{J(\tau_n)}^{bir} \sqcup \orb_{[3,2^{2(m-1)},1^2]}$, where $\orb_{[3,2^{2(m-1)},1^2]}$ is a birationally rigid unipotent class in $G$. 
\end{lem}
\begin{proof}
We have $L_n$ of type ${\sf T}_1{\sf A}_{n-1}$ and $L_n<M_n<G$, where  $\sigma_n^2=1$ and $Z(L_n)=Z(L_n)^\circ \sqcup \hat z Z(L_n)^\circ$ with $Z(L_n)^\circ = \exp(\C \check{\omega}_n)$.
Then 
\begin{align*} 
S_n &= \bigcup_{z \in Z(L_n)^\circ} G \cdot ( z\ind_{L_n}^{C_G(z)} \{1\})=\\
&=G \cdot ((Z(L_n)^\circ)^{reg} )\cup G \cdot ( \sigma_n \ind_{L_n}^{M_n} \{1\}) \cup \ind_{L_n}^G \{1\}
 \end{align*}
where the last two members in the union are the isolated classes in $S_n$.
We show that $\ind_{L_n}^{{M_n}} \{1\}$ is birationally induced.
We have ${M_n}$ of type ${\sf D}_n$, and $u_m=x_{\beta_1}(1)\cdots x_{\beta_m}(1)$ is an element of $\ind_{L_n}^{M_n} \{1\}$.
Then if $K:= M_n$, the  centralizer $C_{\overline K}(\pi_K( u_m))$ is connected by \cite[p. 399]{Carter2}, and the claim follows.

We have $\ind_{L_n}^G \{1\} = \orb_{\mathbf{d}}$ with $\mathbf{d} = [3,2^{2(m-1)},1^2]$ a full-member partition (see \S \ref{sss_bruc}), hence $\orb_{\mathbf{d}}$ is birationally rigid, hence not birationally induced from $(L_n, \{1\})$, and it forms a single birational sheet.
 Therefore
$$
\overline{J(\tau_n)}^{bir} = \bigcup_{\zeta \in \C \setminus \pi i \Z} \orb_{\exp(\zeta \check{\omega}_n)}\sqcup \orb_{\sigma_n u_m}
$$
and $Z(G)\overline{J(\tau_n)}^{bir} =\overline{J(\tau_n)}^{bir} \sqcup\hat z \overline{J(\tau_n)}^{bir}$. Also
$
S_n=\overline{J(\tau_n)}^{bir}
\sqcup Z_m
$.
 \end{proof}

We consider the remaining spherical pseudo-Levi subgroups:
\begin{enumerate}[label=(\roman*)]
\item For $\ell = 2, \dots, n$,  the subgroup $M_\ell$ is maximal of type ${\sf D}_\ell {\sf B}_{n-\ell}$.
If $\ell$ is even we have $\sigma_\ell^2=1$ and  $Z(L)=\langle\sigma_\ell\rangle\times Z(G)$;
if $\ell$ is odd we have $\sigma_\ell^2=\hat z$ and 
$Z(L)=\langle\sigma_\ell\rangle$.
In any case $\sigma_\ell$ and $\hat z\sigma_\ell$ are $G$-conjugate (via the reflection $s_{e_1}$).
Then $\orb_{\sigma_\ell}$
is a (birational) sheet consisting of an isolated class, and $Z(G)\orb_{\sigma_\ell}=\orb_{\sigma_\ell}$.
\item For $\ell = n$, we get $M_n$ maximal of type ${\sf D}_n$. Then $M_n$ admits the birationally rigid unipotent class  $\orb_{u_k}^{M_n}$, 
 corresponding  to the partition $[2^{2k}, 1^{2(n-2k)}]$ in ${\rm SO}_{2n}(\C)$, for $k = 1, \dots, m -1$.
Since $\sigma_n \sim_W \hat z\sigma_n$ via $s_n$ and $\orb_{u_k}^{M_n}$ is characteristic in $M_n$, we have that $Z(G)\orb_{\sigma_\ell u_k}=\orb_{\sigma_\ell u_k}$
is a (birational) sheet consisting of an isolated class, for $k = 1, \dots, m-1$.
\end{enumerate}

Up to central elements, the remaining spherical conjugacy classes in $G$ are $X_\ell$ corresponding to the partition $[2^{2 \ell}, 1^{2n + 1 - 4 \ell}]$, for $\ell=1,\ldots,m$ and $Z_\ell$ corresponding to the partition $[3,2^{2( \ell - 1)}, 1^{2n + 2 - 4 \ell}]$, for $\ell=2,\ldots,m$: these are all birationally rigid unipotent conjugacy classes in $G$.

\begin{table}[H] 
\centering
{\renewcommand{\arraystretch}{1.5}
\begin{tabular}{|c|c|c|c|}
\hline
$\tau$                       & $\overline{J(\tau)}^{bir}$               & $\Lambda(\orb)$ &$d$ \\ \hline
$ (L_1, Z(L_1), \{1\})$ &
 $\bigcup\limits_{\zeta \in \C \setminus 2 \pi i \Z} \orb_{\exp(\zeta \check{\omega}_1)}  \sqcup Z_1\sqcup\hat z  Z_1$ & $2 n_1 \omega_1 +  n_2 \omega_2$ 
& $1$  \\ \hline
$ (L_n, Z(L_n)^\circ, \{1\})$ & 
$\bigcup\limits_{\zeta \in \C \setminus \pi i \Z} \orb_{\exp(\zeta \check{\omega}_n)} \sqcup \orb_{\sigma_n u_m}$ &
$\sum\limits_{i=1}^{n-1} n_i\omega_i + 2n_n\omega_n $ &$2$        \\ \hline
\makecell{$(M_\ell, \{\sigma_\ell\}, \{1\})$\\$ \ell = 2, \dots, m-1$ }&
$\orb_{\sigma_\ell}$ &
$\sum\limits_{i = 1}^{2 \ell -1} 2 n_i \omega_i + n_{2\ell} \omega_{2\ell}$&$1$  \\ \hline
\makecell{$(M_\ell, \{\sigma_\ell\}, \{1\})$\\$ \ell = m+1, \dots, n$} &
$\orb_{\sigma_\ell}$ &
$\sum\limits_{i = 1}^{2(n - \ell)} 2 n_i \omega_i + n_{2(n -\ell) + 1} \omega_{2(n -\ell) + 1}$&$1$  \\ \hline
$(M_m, \{\sigma_m\}, \{1\})$ &
$\orb_{\sigma_m}$ &
$\sum\limits_{i = 1}^n 2 n_i \omega_i$&$1$  \\ \hline
\makecell{$ (M_n, \{\sigma_n\}, \orb^{M_n}_{u_\ell})$\\ $\ell = 1, \dots, m-1$}&
$\orb_{\sigma_n u_\ell}$ &
$\sum\limits_{i = 1}^{2 \ell + 1} n_i \omega_i$&$1$  \\ \hline
\makecell{$(G,\{1\}, \orb_{[2^{2 \ell}, 1^{2n + 1 - 4 \ell}]})$\\$ \ell = 1, \dots, m-1$}&
$X_\ell$ &
$\sum\limits_{i = 1}^{ \ell } n_{2i} \omega_{2i}$&$2$  \\ \hline
$(G,\{1\}, \orb_{[2^{2m}, 1]})$&
$X_m$ &
$\sum\limits_{i = 1}^{m-1} n_{2i} \omega_{2i} + 2 n_n \omega_n$&$2$  \\ \hline
\makecell{$(G,\{1\}, \orb_{[3,2^{2( \ell - 1)}, 1^{2n + 2 - 4 \ell}]})$\\ $\ell = 2, \dots, m-1$}&
$Z_\ell$ &
$\sum\limits_{i = 1}^{ 2\ell } n_i \omega_{i} \mid \sum\limits_{i = 1}^{ \ell }  n_{2i-1} \in 2\N$ &$2$ \\ \hline
$(G,\{1\}, \orb_{[3,2^{2( m- 1)}, 1^2]})$  &
$Z_m$ &
$\sum\limits_{i = 1}^{ n } n_i \omega_{i} \mid \sum\limits_{i = 1}^{ m }  n_{2i-1}, n_n \in 2\N$ &$2$ \\ 
\hline
\end{tabular}
}
\caption{Type $\mathsf{B}_{n}$, $n=2m$, $m \geq 2$.}
\label{tab_bneven}
\end{table}

\begin{prop}
Theorem \ref{thm_principal} holds for $G$ of type $\mathsf{B}_{2m}$, for $m \geq 2$.	
\end{prop}
\begin{proof}
From \cite[Table 10, 11, 12]{Costantini2010}
the weight monoid is preserved along the classes in each $Z(G)\overline{J(\tau)}^{bir}$. The entries in the third column of Table \ref{tab_bneven} are pairwise distinct.
\end{proof}

\subsection{Type $\mathsf{D}_n, n \geq 4$}
The highest root is $\beta=\alpha_1+2\alpha_2+\cdots+2\alpha_{n-2}+\alpha_{n-1}+\alpha_n$.
We fix the following notation:
$$
\hat z_1 =\sigma_1 =\alpha^\vee_{n-1}(-1) \alpha^\vee_{n}(-1),\quad
\hat z_{n-1} =\sigma_{n-1},\quad
\hat z_n =\sigma_n.
$$
Then:
\begin{itemize}[label=-]
\item if $n=2m$ is even, $\hat z_{n-1}$ and $\hat z_n$ are involutions and $\hat z_n\hat z_{n-1}=\hat z_1$; in particular, $$\prod\limits_{j=0}^{m-1} \alpha^\vee_{2j +1}(-1) = \begin{cases}\hat{z}_n & m \mbox{ even}\\
\hat{z}_{n-1} & m \mbox{ odd} \end{cases},$$
hence $Z(G) =\langle \hat{z}_1, \hat{z}_n \rangle = \langle \hat{z}_1, \hat{z}_{n-1} \rangle \simeq \Z_2 \times \Z_2$.
\item If $n=2m+1$ is odd, $\hat z_n=\hat z_{n-1}^{-1}$
has order 4 and $\hat z_n^2=\hat z_1$, hence $Z(G) =\langle \hat{z}_n \rangle = \langle \hat{z}_{n-1} \rangle \simeq \Z_4$.
\end{itemize}

The following result holds for any $n \geq 4$:

\begin{lem}
Let $\tau_1 := (L_1, Z(L_1)^\circ, \{1\}) \in \mathscr{D}(G)$. Then the sheet $S_1 := \overline{J(\tau_1)}^{reg} = \overline{J(\tau_1)}^{bir}$ is a spherical birational sheet containing the unipotent class $\orb_{\mathbf{d}}$, with $\mathbf{d} = [3,1^{2n-3}]$.
\end{lem}
\begin{proof}
$L_1$ is maximal of type ${\sf T}_1 {\sf D}_{n-1}$ and $Z(L_1)^\circ=\exp(\C\check\omega_1)$, $Z(L_1)=Z(L_1)^\circ\sqcup Z(L_1)^\circ \hat z_n$.
We have
$$
S_1 = \bigcup_{z \in Z(L_1)^\circ} G \cdot (z \ind_{L_1}^{C_G(z)} \{1\}) = G \cdot ((Z(L_1)^\circ)^{reg}) \cup  \ind_{L_1}^{G} \{1\} \cup\hat z_1 \ind_{L_1}^{G} \{1\}.
$$
Then the only unipotent isolated class in $S_1$ is $\ind_{L_1}^{G} \{1\} = \orb_{\mathbf{d}}$, where $\mathbf{d} =[3,1^{2n-3}]$  in $\SO_{2n}(\C)$. 
We show that $ \ind_{L_1}^G\{1\}$ is birationally induced from $(L_1,\{1\})$. Let $u \in \ind_{L_1}^G \{1\}$, the  centralizer $C_{\overline G}(\pi(u))$ is connected by \cite[p. 399]{Carter2}, and the claim follows.
Moreover the class $\orb_{[3,1^{2n-3}]}$ is the class denoted by $Z_{1}$ in \cite{Costantini2010}. Therefore
$$
{S}_1 = \bigcup_{\zeta \in \C \setminus 2\pi i \Z} \orb_{\exp(\zeta \check{\omega}_1)}\sqcup Z_{1}\sqcup \hat z_1 Z_{1}
$$
and $Z(G)S_{1}=S_{1}\sqcup \hat z_n S_1$
\end{proof}

\subsubsection{Type $\mathsf{D}_{2m}$, $m \geq 2$}
Let $\vartheta$ denote the graph automorphism of $G$ swapping $\alpha_{n-1}$ and $\alpha_{n}$.

\begin{lem}
The following spherical sheets of $G$ are spherical birational sheets.
\begin{enumerate}
\item[(i)]$S_n := \overline{J(\tau_n)}^{reg} = \overline{J(\tau_n)}^{bir}$, where $\tau_n := (L_n, Z(L_n)^\circ, \{1\})$;
\item[(ii)] $S_{n-1} := \vartheta(S_n) = \overline{J(\tau_{n-1})}^{reg} = \overline{J(\tau_{n-1})}^{bir}$, where $\tau_{n-1} := (L_{n-1}, Z(L_{n-1})^\circ, \{1\})$,
\end{enumerate}
\end{lem}

\begin{proof}
Consider $L_n$  of type ${\sf T}_1 {\sf A}_{n-1}$ and maximal; moreover, $Z(L_n)^\circ=\exp(\C\check\omega_n)$ and $Z(L_n)=Z(L_n)^\circ\sqcup Z(L_n)^\circ \hat z_1$. We have
$$
S_n = \bigcup_{z \in Z(L_{n})^\circ} G \cdot (z \ind_{L_n}^{C_G(z)} \{1\}) = G \cdot ((Z(L_n)^\circ)^{reg}) \cup \ind_{L_n}^G \{1\} \cup\hat z_n \ind_{L_n}^G \{1\}.
$$
Let $\orb = \ind_{L_n}^G \{1\}$, then $\orb$ is one of the two unipotent classes corresponding to the very even partition $[2^n]$, the one denoted by $X_m$ in \cite{Costantini2010}.
We show that $ \ind_{L_n}^G\{1\}$ is birationally induced from $(L_n,\{1\})$.
Let $u \in \ind_{L_n}^G \{1\}$, then
 $C_{\overline G}(\pi(u))$ is connected by \cite[p. 399]{Carter2}, and the claim follows. 

Therefore
$$
{S}_n = \bigcup_{\zeta \in \C \setminus 2\pi i \Z} \orb_{\exp(\zeta \check{\omega}_n)}\sqcup X_{m}\sqcup \hat z_n X_{m}
$$
is a spherical birational sheet
and $Z(G)S_{n}=S_{n}\sqcup \hat z_1 S_n$.
By applying the automorphism $\vartheta$ we get
$$
{S}_{n-1} = \bigcup_{\zeta \in \C \setminus 2\pi i \Z} \orb_{\exp(\zeta \check{\omega}_{n-1})}\sqcup X_{m}'\sqcup \hat z_{n-1} X_{m}'
$$
and $Z(G)S_{n-1}=S_{n-1}\sqcup \hat z_1 S_{n-1}$.
\end{proof}

We consider the remaining spherical pseudo-Levi subgroups.
For $\ell = 2, \dots, m$, 
the subgroup $M_\ell$ is maximal of type ${\sf D}_\ell {\sf D}_{n-\ell}$.
If $\ell$ is even we have $\sigma_\ell^2=1$ and 
$Z(M_\ell)=\langle\sigma_\ell\rangle\times Z(G)$;
if $\ell$ is odd we have $\sigma_\ell^2=\hat z_1$ and 
$Z(M_\ell)=\langle\sigma_\ell\rangle\times \langle z_n \rangle$. 
For $\ell = 2, \dots, m-1$,  $\sigma_\ell$ 
is not $G$-conjugate to $\hat z_n\sigma_\ell$ and $\hat z_{n-1}\sigma_\ell$, as one can see by passing to ${\rm SO}_{2n}(\C)$.
On the other hand, for $\ell = 2, \dots, m$, $\omega_\ell$ is $W$-conjugate to $\omega_\ell-2\omega_1$, therefore 
$\sigma_\ell$ 
is $G$-conjugate to $\hat z_1\sigma_\ell$.
Moreover, $\omega_m$ is $W$-conjugate to $\omega_m-2\omega_1$, $\omega_m-2\omega_{n-1}$, $\omega_m-2\omega_n$, hence
$\sigma_m$ is $G$-conjugate to $\hat z_1\sigma_m$, $\hat z_{n-1}\sigma_m$ and $\hat z_n\sigma_m$.
Then $\orb_{\sigma_\ell}$
is a (birational) sheet consisting of an isolated class, and $Z(G)\orb_{\sigma_\ell} = \orb_{\sigma_\ell} \sqcup \hat z_n \orb_{\sigma_\ell}$ for $\ell=2,\ldots,m-1$, whereas $Z(G) \orb_{\sigma_m}= \orb_{\sigma_m}$.

Up to central elements, the remaining spherical conjugacy classes in $G$ are $X_\ell$ corresponding to the partition $[2^{2\ell}, 1^{2n - 4\ell}]$, for $\ell=1,\ldots,m-1$ and $Z_\ell$ corresponding to the partition $[3, 2^{2(\ell-1)}, 1^{2n - 4\ell +1}]$, for $\ell=2,\ldots,m$: these are all birationally rigid unipotent conjugacy classes in $G$. 

\begin{table}[H] 
\centering
{\renewcommand{\arraystretch}{1.5}
\begin{tabular}{|c|c|c|c|}
\hline
$\tau$                       & $\overline{J(\tau)}^{bir}$   &               $\Lambda(\orb)$ & $d$ \\ \hline
$ (L_1, Z(L_1)^\circ, \{1\})$      & 
$\bigcup\limits_{\zeta \in \C \setminus 2\pi i \Z} \orb_{\exp(\zeta \check{\omega}_1)}\sqcup Z_{1}\sqcup \hat z_1 Z_{1}$
  & $2 n_1 \omega_1 +  n_2 \omega_2$&2 \\ \hline
$ (L_n, Z(L_n)^\circ, \{1\})$ &
$\bigcup\limits_{\zeta \in \C \setminus 2\pi i \Z} \orb_{\exp(\zeta \check{\omega}_n)}\sqcup X_{m}\sqcup \hat z_n X_{m}$
& $\sum\limits_{i=1}^{m-1} n_{2i}\omega_{2i} +  2n_n \omega_n$ &2\\ \hline
$ (L_{n-1}, Z(L_{n-1})^\circ, \{1\})$ &
$\bigcup\limits_{\zeta \in \C \setminus 2\pi i \Z} \orb_{\exp(\zeta \check{\omega}_{n-1})}\sqcup X_{m}'\sqcup \hat z_{n-1} X_{m}'$
& $\sum\limits_{i=1}^{m-1} n_{2i} \omega_{2i} +  2n_{n-1} \omega_{n-1}$ &2\\ \hline
\makecell{$ (M_\ell, \{\sigma_{\ell}\}, \{1\})$\\ $\ell = 2, \dots, m-1$ }& 
$\orb_{\sigma_\ell}$ &    
$\sum\limits_{i=1}^{2\ell -1} 2n_i \omega_{i} + n_{2\ell} \omega_{2\ell}$ &2 \\ \hline
$  (M_m, \{\sigma_{m}\}, \{1\})$ & 
$\orb_{\sigma_m}$ &  
$\sum\limits_{i=1}^{n} 2n_i \omega_{i}$  &1  \\ \hline
\makecell{$(G, \{1\}, \orb_{[2^{2\ell}, 1^{2n - 4\ell}]})$\\$  \ell=1,\ldots,m-1$ }& 
$X_\ell$ &        
$\sum\limits_{i=1}^{\ell} n_{2i}\omega_{2i}$ &4 \\ \hline
\makecell{$(G, \{1\}, \orb_{[3, 2^{2(\ell-1)}, 1^{2n - 4\ell +1}]})$\\ $\ell = 2, \dots, m-1 $} & $Z_\ell$ & $\sum\limits_{i=1}^{2\ell} n_i \omega_{i} \mid \sum\limits_{i=1}^{\ell} n_{2i -1} \in 2\N$ 
& 4  	\\ \hline
$(G, \{1\}, \orb_{[3, 2^{2(m-1)}, 1]})$& $Z_m$ & 
\begin{tabular}{c}
$\sum\limits_{i=1}^{n} n_i \omega_{i} \mid $\\
$\sum\limits_{i=1}^m n_{2i -1}, n_{n-1} + n_n \in 2\N$
\end{tabular}
& 4  	\\ \hline
  \end{tabular}
}
\caption{Type $\mathsf{D}_{n}, n=2m, m \geq 2$.} \label{tab_dneven}
\end{table}

\begin{prop}
Theorem \ref{thm_principal} holds for $G$ of type $\mathsf{D}_{2m}$, with $m \geq 2$.
\end{prop}
\begin{proof}
From \cite[Table 6, 7]{Costantini2010}
the weight monoid is preserved along the classes in each $Z(G)\overline{J(\tau)}^{bir}$. The entries in the third column of Table \ref{tab_dneven} are pairwise distinct.
\end{proof}

\subsubsection{Type $\mathsf{D}_{2m+1}$, $m \geq 2$}

\begin{lem}
Let $ \tau_n = (L_n, Z(L_n), \{1\}) \in \mathscr{D}(G)$.
The spherical sheet $S_n:= \overline{J(\tau_n)}^{reg}$ is a birational spherical sheet, containing the unipotent class $\orb_{\mathbf{d}}$, with $\mathbf{d} = [2^{n-1},1^2]$.
\end{lem}
\begin{proof}
 Consider $L_n$: it is maximal of type ${\sf T}_1 {\sf A}_{n-1}$ and $Z(L_n)=\exp(\C\check\omega_n)$ is connected.  We have:
\begin{align*} 
S_n &= \bigcup_{z \in Z(L_{n})} G \cdot (z \ind_{L_n}^{C_G(z)} \{1\})=\\
&=G \cdot (Z(L_n)^{reg}) \cup \ind_{L_n}^G \{1\} \cup\hat z_1 \ind_{L_n}^G \{1\}\cup\hat z_{n-1} \ind_{L_n}^G \{1\}\cup\hat z_n \ind_{L_n}^G \{1\}. \end{align*}
Let $\orb = \ind_{L_n}^G \{1\}$, then $\orb$ is the unipotent class corresponding to the partition $[2^{n-1},1^2]$ in $\SO_{2n}(\C)$, the unipotent class denoted by $X_m$ in \cite{Costantini2010}.
We show that $ \ind_{L_n}^G\{1\}$ is birationally induced from $(L_n,\{1\})$. Let $u \in \ind_{L_n}^G \{1\}$, then $C_{\overline G}(\pi(u))$ is connected by \cite[p. 399]{Carter2}, and the claim follows.
Therefore
$$
{S}_n = \bigcup_{\zeta \in \C \setminus 2\pi i \Z} \orb_{\exp(\zeta \check{\omega}_n)}\sqcup X_{m}\sqcup \hat z_1 X_{m}\sqcup\hat z_{n-1} X_{m}\sqcup\hat z_n X_{m}
$$
is a spherical birational sheet
and $Z(G)S_{n}=S_{n}$
\end{proof}

We consider the remaining spherical pseudo-Levi subgroups.
For $\ell = 2, \dots, m$,   $M_\ell$ is maximal of type ${\sf D}_\ell {\sf D}_{n-\ell}$. 
 If $\ell$ is even we have $\sigma_\ell^2=1$ and 
$Z(M_\ell)=\langle\sigma_\ell\rangle\times Z(G)\simeq \Z_4\times \Z_2$; if $\ell$ is odd we have $\sigma_\ell^2=\hat z_1$ and 
$Z(M_\ell)=\langle\sigma_\ell, z_n \rangle\simeq \Z_4\times \Z_2$. 
For $\ell = 2, \dots, m$,  $\sigma_\ell$ 
is not $G$-conjugate to $\hat z_n\sigma_\ell$ (and $\hat z_{n-1}\sigma_\ell$), as one can see by passing to ${\rm SO}_{2n}(\C)$.
On the other hand, for $\ell = 2, \dots, m$, $\omega_\ell$ is $W$-conjugate to $\omega_\ell-2\omega_1$ and therefore 
$\sigma_\ell$ 
is $G$-conjugate to $\hat z_1\sigma_\ell$. 
Then $\orb_{\sigma_\ell}$
is a (birational) sheet consisting of an isolated class, and $Z(G) \orb_{\sigma_\ell} = \orb_{\sigma_\ell} \cup \hat z_n \orb_{\sigma_\ell}$ for $\ell=2,\ldots,m$.

Up to central elements, the remaining spherical conjugacy classes in $G$ are $X_\ell$ corresponding to the partition $[2^{2 \ell}, 1^{2n - 4 \ell}]$, for $\ell=1,\ldots,m-1$ and $Z_\ell$ corresponding to the partition $[3,2^{2( \ell - 1)}, 1^{2n + 1 - 4 \ell}]$, for $\ell=2,\ldots,m$: these are all birationally rigid unipotent conjugacy classes in $G$. 

\begin{table}[H] 
\centering
{\renewcommand{\arraystretch}{1.5}
\begin{tabular}{|c|c|c|c|}
\hline
$\tau$                       &$\overline{J(\tau)}^{bir}$  &               $\Lambda(\orb)$ & $d$ \\ \hline
$ (L_1, Z(L_1)^\circ, \{1\})$      & 
$\bigcup\limits_{\zeta \in \C \setminus 2\pi i \Z} \orb_{\exp(\zeta \check{\omega}_1)}\sqcup Z_{1}\sqcup \hat z_1 Z_{1}$
  & $2 n_1\omega_1 +  n_2 \omega_2$&2 \\ \hline
$ (L_n, Z(L_n), \{1\})$ &
\begin{tabular}{c}
$\bigcup\limits_{\zeta \in \C \setminus 2\pi i \Z} \orb_{\exp(\zeta \check{\omega}_n)}\sqcup$\\
$\sqcup X_{m}\sqcup \hat z_1 X_{m}\sqcup\hat z_{n-1} X_{m}\sqcup\hat z_n X_{m}$
 \end{tabular}
& $\sum\limits_{i=1}^{m-1} n_{2i} \omega_{2i} +  n_{n-1}(\omega_{n-1} + \omega_n)$ &1\\ \hline
\makecell{$ (M_\ell, \{\sigma_{\ell}\}, \{1\})$\\$\ell = 2, \dots, m-1$} & $\orb_{\sigma_\ell}$ &    
$\sum\limits_{i=1}^{2\ell -1} 2n_i \omega_{i} + n_{2\ell} \omega_{2 \ell}$ &2 \\ \hline
$  (M_m, \{\sigma_{m}\}, \{1\}$ & 
$\orb_{\sigma_m}$ &  
$\sum\limits_{i=1}^{n-2} 2n_i \omega_{i} + n_{n-1}(\omega_{n-1} + \omega_n)$  &2  \\ \hline
\makecell{$(G, \{1\}, \orb_{[2^{2\ell}, 1^{2n - 4\ell}]})$\\$  \ell=1,\ldots,m-1$ }& 
$X_\ell$ &        
$\sum\limits_{i=1}^{\ell} n_{2i}\omega_{2i}$ &4 \\ \hline
\makecell{$(G, \{1\}, \orb_{[3, 2^{2(\ell-1)}, 1^{2n - 4\ell +1}]})$\\ $\ell = 2, \dots, m-1 $ }& $Z_\ell$ & $\sum\limits_{i=1}^{2\ell} n_i \omega_{i} \mid \sum\limits_{i=1}^{\ell} n_{2i -1} \in 2\N$ 
& 4  	\\ \hline
$(G, \{1\}, \orb_{[3, 2^{2(m-1)}, 1^3]})$& $Z_m$ & 
\begin{tabular}{c}
$\sum\limits_{i=1}^{n-2} n_i \omega_{i} + n_{n-1}(\omega_{n-1} + \omega_n) \mid $\\
$\sum\limits_{i=1}^{m} n_{2i -1} \in 2\N$
\end{tabular}
& 4  	\\ \hline
  \end{tabular}
  }
\caption{Type $\mathsf{D}_{n}, n=2m+1, m \geq 2$.}    \label{tab_dnodd}
\end{table}

\begin{prop} Theorem \ref{thm_principal} holds for $G$ of type $\mathsf{D}_{2m+1}$, with $m \geq 2$.
\end{prop}
\begin{proof}
From \cite[Table 8, 9]{Costantini2010}
the weight monoid is preserved along the classes in each $Z(G)\overline{J(\tau)}^{bir}$. The entries in the third column of Table \ref{tab_dnodd} are pairwise distinct.
\end{proof}

\subsection{Type $\mathsf{E}_6$}

The highest root is $\beta = (1,2,2,3,2,1)$. We have
$Z(G)=\langle\hat{z}\rangle$, $\hat{z} = \alpha^\vee_1(\xi)\alpha^\vee_6(\xi^{-1})\alpha^\vee_3(\xi^{-1})\alpha^\vee_5(\xi)$ where $\xi$ is a primitive third root of $1$.
\begin{lem}
Let $\tau_1 = (L_1, Z(L_1), \{1\})$. Then the spherical sheet $S_1 := \overline{J(\tau_1)}^{reg}$ is a spherical birational sheet
containing the unipotent class $2A_1$.
\end{lem}
\begin{proof}
$L_1$ is maximal of type ${\sf D}_5 {\sf T}_1$ and $Z(L_1)$ is connected since $L_1=C_G(\exp \check \omega_1)$ and $\exp( 2\pi i\check \omega_1)=\hat z$. 
Then
$$
S_1 = \bigcup_{z \in Z(L_1)} G \cdot (z \ind_{L_1}^G \{1\}) = \bigcup_{\zeta \in \C \setminus 2 \pi i \Z} \orb_{\exp(\zeta \check{\omega}_1)} \cup Z(G) \ind_{L_1}^G \{1\}.
$$
The class $\ind_{L_1}^G \{1\}= 2A_1$ is birationally induced from $(L_1, \{1\})$ by Lemma \ref{lem_compgp}.
Indeed, for $u \in \ind_{L_1}^G \{1\}$, the subgroup $C_{\overline{G}} (\bar{u})$ is connected, by \cite[p. 402]{Carter2}. Hence
$$
S_1 =\overline{J(\tau_1)}^{bir}  = \bigcup_{\zeta \in \C \setminus 2 \pi i \Z} \orb_{\exp(\zeta \check{\omega}_1)} \sqcup Z(G) {2A_1}$$
and $Z(G)\overline{J(\tau_1)}^{bir}=\overline{J(\tau_1)}^{bir}$.
\end{proof}

There is only one more spherical pseudo-Levi subgroup,  $M_2$ of type ${\sf A}_1 {\sf A}_5$.
Observe that $\sigma_2^2=1$ and $Z(M_2)={\langle \sigma_2\rangle}\times Z(G)$.
$M_2$ gives rise to the (birational) sheet $\orb_{\sigma_2}$ which coincides with an isolated class.
We have $Z(G)\orb_{\sigma_2} = \orb_{\sigma_2} \sqcup \hat z  \orb_{\sigma_2} \sqcup \hat z^2 \orb_{\sigma_2}$.
Up to central elements, the remaining spherical conjugacy classes in $G$ are $A_1$ and $3A_1$: these are birationally rigid unipotent conjugacy classes in $G$.

\begin{table}[H]
\centering
{\renewcommand{\arraystretch}{1.5}
\begin{tabular}{|c|c|c|c|}
\hline
$\tau$                       & $\overline{J(\tau)}^{bir}$              & $\Lambda(\orb)$& $d$ \\ \hline
$ (L_1, Z(L_1), \{1\})$      & 
$\bigcup\limits_{\zeta \in \C \setminus 2 \pi i \Z} 
\orb_{\exp(\zeta \check{\omega}_1)}
 \sqcup Z(G) {2A_1}$
  & $ n_1 (\omega_1 + \omega_6)+n_2\omega_2$&1 \\ \hline
  $(M_2,\{\sigma_2\},\{1\})$
 &
$\orb_{\sigma_2}$ 
 &
 $n_1(\omega_1 + \omega_6) +n_3(\omega_3 + \omega_5) +2n_2 \omega_2 +2 n_4 \omega_4$    &3    \\ \hline
 $(G,\{1\}, A_1)$
 &
 ${A_1}$ &
 $n_2 \omega_2$ &3 \\ \hline
$(G,\{1\}, 3A_1)$
& ${3A_1}$ 
& $n_1 (\omega_1 + \omega_6) +n_3(\omega_3 + \omega_5) +n_2 \omega_2 +n_4 \omega_4$   &3 	\\ \hline
  \end{tabular}}
    \caption{Type $\mathsf{E}_6$.}
  \label{tab_e6}
\end{table}

\begin{prop}
Theorem \ref{thm_principal} holds for $G$ of type $\mathsf{E}_6$.
\end{prop}
\begin{proof}
From \cite[Table 16, 17]{Costantini2010}
the weight monoid is preserved along the classes in each $Z(G)\overline{J(\tau)}^{bir}$. The entries in the third column of Table \ref{tab_e6} are pairwise distinct.
\end{proof}

\subsection{Type $\mathsf{E}_7$}

The highest root is $\beta = (2,2,3,4,3,2,1)$. We have
$Z(G)=\langle\hat{z}\rangle$, $\hat{z} = \alpha^\vee_2(-1)\alpha^\vee_5(-1)\alpha^\vee_7(-1)$.

\begin{lem}
Let $\tau_7 = (L_7, Z(L_7), \{1\})$.
Then the spherical sheet $S_7 := \overline{J(\tau_7)}^{reg}$ is a spherical birational sheet
containing the unipotent class $(3A_1)''$.
\end{lem}
\begin{proof}
$L_7$ is maximal of type ${\sf E}_6{\sf T}_1$ and $Z(L_7)$ is connected since $L_7=C_G(\exp \check \omega_7)$ and $\exp (2\pi i\check \omega_7)=\hat z$.
Then
$$
S_7 = \bigcup_{z \in Z(L_7)} G \cdot (z \ind_{L_7}^G \{1\}) = \bigcup_{\zeta \in \C \setminus 2 \pi i \Z} \orb_{\exp(\zeta \check{\omega}_7)} \cup Z(G) \ind_{L_7}^G \{1\}.
$$
The isolated class $\ind_{L_7}^G \{1\}=(3A_1)''$ is birationally induced:  for $u \in \ind_{L_7}^G \{1\}$, the group $C_{\overline{G}} (\pi(u))$ is connected, by \cite[p. 403]{Carter2}.
Hence
$$
 S_7 = \overline{J(\tau_7)}^{bir} = \bigcup_{\zeta \in \C \setminus 2 \pi i \Z} \orb_{\exp(\zeta \check{\omega}_7)} \sqcup Z(G) {(3A_1)''}$$
and $Z(G)\overline{J(\tau_7)}^{bir}=\overline{J(\tau_7)}^{bir}$.
\end{proof}

We consider the remaining spherical pseudo-Levi subgroups:
\begin{enumerate}[label=(\roman*)]
\item  Consider
$M_2$, maximal of type ${\sf A}_7$.
We have $\sigma_2^2=\hat z$ and 
$Z(M_2)=\langle\sigma_2\rangle$, $\sigma_2$ and $\hat z\sigma_2=\sigma_2^{-1}$ are $G$-conjugate via  $w_0$.
Then $\orb_{ \sigma_2}$
is a (birational) sheet consisting of an isolated class, and $Z(G)S=S$.
\item  Consider $M_1$, maximal of type ${\sf D}_6{\sf A}_1$.
 We have $\sigma_1^2=1$ and 
$Z(M_1)=\langle {\hat z, \sigma_1}\rangle$, $\sigma_1$ and $\hat z\sigma_1$ are not $G$-conjugate (in fact $G$ has $2$ classes of non-central involutions: $\orb_{\sigma_1}$ and $\orb_{\hat z\sigma_1}$).
Then $\orb_{ \sigma_1}$
is a (birational) sheet consisting of an isolated class, and $Z(G)\orb_{ \sigma_1}=\orb_{ \sigma_1}\sqcup\hat z \orb_{ \sigma_1}$.
\end{enumerate}
Up to central elements, the remaining spherical conjugacy classes in $G$ are $A_1$, $2A_1$, $(3A_1)'$ and $4A_1$: these are birationally rigid unipotent conjugacy classes in $G$.

\begin{table}[H]
\centering
{\renewcommand{\arraystretch}{1.5}
\begin{tabular}{|c|c|c|c|}
\hline
$\tau$                       & $\overline{J(\tau)}^{bir}$                & $\Lambda(\orb)$& $d$ \\ \hline
$ (L_7, Z(L_7), \{1\})$      & 
$\bigcup\limits_{\zeta \in \C \setminus 2 \pi i \Z} 
\orb_{\exp(\zeta \check{\omega}_7)}
 \sqcup Z(G) {(3 A_1)''}$
  & $ n_1 \omega_1 +n_6 \omega_6 +2n_7\omega_7$ & 1 \\ \hline
$ (M_1, \{\sigma_1\}, \{1\})$ &
$\orb_{\sigma_1} $
&$2 n_1 \omega_1 + 2 n_3\omega_3 +n_4 \omega_4 +n_6 \omega_6$
& 2  \\ \hline
$(M_2, \{\sigma_2\},\{1\})$&
$\orb_{\sigma_2}$
 &
 $\sum\limits_{i=1}^7 2n_i \omega_i$& 1 
 \\ \hline
$(G,\{1\}, A_1)$&
${A_1}$ &
 $n_1 \omega_1$ & 2  \\ \hline
$(G,\{1\}, 2A_1)$&
${2A_1}$ &
 $n_1 \omega_1 +n_6 \omega_6$ & 2  \\ \hline
$(G,\{1\}, {(3A_1)'})$&
${(3A_1)'}$ &
 $n_1 \omega_1 +n_3 \omega_3 +n_4\omega_4 +n_6 \omega_6 $& 2   \\ \hline
$(G,\{1\}, 4A_1)$&
$4 A_1$ &
 $\sum\limits_{i=1}^7 n_i \omega_i \mid n_2 + n_5 + n_7$ even & 2   \\ \hline
  \end{tabular}}
    \caption{Type $\mathsf{E}_7$.}
  \label{tab_e7}
\end{table}

\begin{prop}
Theorem \ref{thm_principal} holds for $G$ of type $\mathsf{E}_7$.
\end{prop}
\begin{proof}
From \cite[Table 18, 19]{Costantini2010}
the weight monoid is preserved along  the classes in each $Z(G)\overline{J(\tau)}^{bir}$. The entries in the third column of Table \ref{tab_e7} are pairwise distinct.
\end{proof}

\subsection{Type $\mathsf{E}_8$}

The highest root is $\beta_1 =   (2,3,4,6,5,4,3,2)$. There are no spherical proper Levi subgroups.

We list the spherical pseudo-Levi subgroups.
\begin{enumerate}[label=(\roman*)]
\item  Consider $M_8$, maximal of type ${\sf A}_1{\sf E}_7$.
We have $\sigma_8^2=1$ and 
$Z(M_8)=\langle\sigma_8\rangle$.
Then $\orb_{\sigma_8}$
is a (birational) sheet consisting of an isolated class.
\item Consider $M_1$, maximal of type ${\sf D}_8$.
We have $\sigma_1^2=1$ and 
$Z(M_1)=\langle\sigma_1\rangle$.
Then $\orb_{ \sigma_1}$
is a (birational) sheet consisting of an isolated class.
\end{enumerate}
The remaining spherical conjugacy classes in $G$ are $A_1$, $2A_1$, $3A_1$ and $4A_1$: these are birationally rigid unipotent conjugacy classes in $G$.

\begin{table}[H]
\centering
{\renewcommand{\arraystretch}{1.5}
\begin{tabular}{|c|c|c|} 
\hline
$\tau$                       & $\overline{J(\tau)}^{bir}$                & $\Lambda(\orb)$ \\ \hline
$(M_8, \{\sigma_8\}, \{1\})$    & 
 $\orb_{\sigma_8}$
  & $ n_1 \omega_1 +n_6 \omega_6 +2n_7 \omega_7 + 2 n_8 \omega_8$\\ \hline
$(M_1,\{\sigma_1\},\{1\})$& 
$\orb_{\sigma_1}$ 
 &
 $\sum\limits_{i =1}^8 2 n_i \omega_i $        \\ \hline
$(G,\{1\}, A_1)$	 & 
${A_1}$ 
 &
 $ n_8 \omega_8 $        \\ \hline
$(G,\{1\}, 2A_1)$									& 
${2A_1}$  			&
$ n_1 \omega_1 +n_8 \omega_8 $        		\\ \hline
$(G,\{1\}, 3A_1)$									& 
${3A_1}$  			&
$ n_1 \omega_1 +n_6\omega_6 +n_7 \omega_7 +n_8 \omega_8 $        		\\ \hline
$(G,\{1\}, 4A_1)$								& 
${4A_1}$  			&
$ \sum\limits_{i=1}^8 n_i \omega_i$  		\\ \hline
  \end{tabular}}
  \caption{Type $\mathsf{E}_8$.}
  \label{tab_e8}
\end{table}

\begin{prop}
Theorem \ref{thm_principal} holds for $G$ of type $\mathsf{E}_8$.
\end{prop}
\begin{proof}
From \cite[Table 20, 21]{Costantini2010}
the weight monoid is preserved along the classes in each $Z(G)\overline{J(\tau)}^{bir}$. The entries in the third column of Table \ref{tab_e8} are pairwise distinct.
\end{proof}

\subsection{Type $\mathsf{F}_4$}

The highest root is $\beta_1 =  (2,3,4,2)$. There are no spherical proper Levi subgroups.

We list the spherical pseudo-Levi subgroups.
\begin{enumerate}[label=(\roman*)]
\item  Consider $M_1$, maximal of type ${\sf A}_1 {\sf C}_3$. 
We have $\sigma_1^2=1$ and 
$Z(M_1)=\langle\sigma_1\rangle$.
Then $\orb_{ \sigma_1}$
is a (birational) sheet consisting of an isolated class.
\item Consider
 $M_4$, maximal of type ${\sf B}_4$.
 We have $\sigma_4^2=1$ and 
$Z(M_4)=\langle\sigma_4\rangle$.
Then $\orb_{\sigma_4}$
is a (birational) sheet consisting of an isolated class.
\item $M_4$ admits the birationally rigid unipotent class $\orb_{x_{\beta_1}(1)}^{M_4}$, corresponding to the partition $[2^2,1^5]$ in ${\rm SO}_9(\C)$.
 Then $ \orb_{ \sigma_4 x_{\beta_1}(1)}$
is a (birational) sheet consisting of an isolated class. 
\end{enumerate}
The remaining spherical conjugacy classes in $G$ are $A_1$, $\widetilde A_1$ and $A_1+\widetilde A_1$: these are birationally rigid unipotent conjugacy classes in $G$.

\begin{table}[H]
\centering
{\renewcommand{\arraystretch}{1.5}
\begin{tabular}{|c|c|c|}
\hline
$\tau$                       & $\overline{J(\tau)}^{bir}$                 & $\Lambda(\orb)$ \\ \hline
$(M_4, \{\sigma_4\}, \{1\})$ & 
$\orb_{\sigma_4}$ 
 &
 $n_4 \omega_4$        \\ \hline
$(M_1, \{\sigma_1\}, \{1\})$ & 
$\orb_{\sigma_1}$  &
 $\sum\limits_{i=1}^{4} 2n_i\omega_i$        \\  \hline
$(M_4, \{\sigma_4\}, \orb_{[2^2,1^5]}^{M_4})$& 
$\orb_{\sigma_4 x_{\beta_1}(1)}$  &
 $\sum\limits_{i=1}^{4} n_i \omega_i$        \\ \hline
$(G,\{1\}, A_1)$								& 
${A_1}$  			&
$ n_1\omega_1$        		\\ \hline
$(G,\{1\}, \widetilde A_1)$							& 
${\widetilde{A}_1}$  			&
$ n_1 \omega_1 + 2 n_4 \omega_4 $        	\\ \hline
$(G,\{1\}, A_1 + \widetilde{A}_1)$								& 
${A_1 + \widetilde{A}_1}$  			&
$ n_1 \omega_1 +n_2 \omega_2 +2n_3\omega_3 +2n_4\omega_4 $        \\ \hline
  \end{tabular}}
    \caption{Type $\mathsf{F}_4$.}
  \label{tab_f4}
\end{table}

\begin{prop}
Theorem \ref{thm_principal} holds for $G$ of type $\mathsf{F}_4$.
\end{prop}
\begin{proof}
From \cite[Table 22, 23, 24]{Costantini2010}
the weight monoid is preserved along the classes in each $Z(G)\overline{J(\tau)}^{bir}$. The entries in the third column of Table \ref{tab_f4} are pairwise distinct.
\end{proof}

\subsection{Type $\mathsf{G}_2$}

The highest root is $\beta_1 =  (3,2)$. There are no spherical proper Levi subgroups.

We list the  spherical pseudo-Levi subgroups.
\begin{enumerate}[label=(\roman*)]
\item Consider $M_2$, maximal of type ${\sf A}_1 {\sf \tilde A}_1$.
We have $\sigma_2^2=1$ and 
$Z(M_2)=\langle\sigma_2\rangle$.
Then
$\orb_{ \sigma_2}$
is a (birational) sheet consisting of an isolated class.
\item  Consider $M_1$, maximal of type ${\sf A}_2$.
 We have $\sigma_1^3=1$  and
$Z(M_1)=\langle\sigma_1\rangle$; moreover, $\sigma_1$ and $\sigma_1^{-1}$ are $G$-conjugate.
Then
$\orb_{ \sigma_1}$
is a (birational) sheet consisting of an isolated class.
\end{enumerate}

The remaining spherical conjugacy classes in $G$ are $A_1$, $\widetilde A_1$: these are birationally rigid unipotent conjugacy classes in $G$. 
\begin{table}[H]
\centering
{\renewcommand{\arraystretch}{1.5}
\begin{tabular}{|c|c|c|}
\hline
$\tau$  & $\overline{J(\tau)}^{bir}$ & $\Lambda(\orb)$ \\ \hline
$(M_2, \{\sigma_2\}, \{1\})$ & 
$\orb_{\sigma_2}$ 
 &
 $2 n_1\omega_1 + 2 n_2 \omega_2$        \\ \hline
$(M_1, \{\sigma_1\}, \{1\})$& 
$\orb_{\sigma_1}$ &
 $n_1 \omega_1$        \\ \hline
				$(G,\{1\}, A_1)$				& 
${A_1}$ 			&
$ n_2 \omega_2$        		\\ \hline
$(G,\{1\}, \widetilde{A}_1)$							& 
${\widetilde{A}_1}$		&
$ n_1 \omega_1 + 2n_2\omega_2 $        	\\ \hline
  \end{tabular}
  }
    \caption{Type $\mathsf{G}_2$.}
  \label{tab_g2}
\end{table}

\begin{prop}
Theorem \ref{thm_principal} holds for $G$ of type $\mathsf{G}_2$.
\end{prop}
\begin{proof}
From \cite[Table 25, 26]{Costantini2010}
the weight monoid is preserved along the classes in each $Z(G)\overline{J(\tau)}^{bir}$. The entries in the third column of Table \ref{tab_g2} are pairwise distinct.
\end{proof}

The proof of Theorem \ref{thm_principal} is complete.

\bigskip
We conclude this Section with another characterization of spherical birational sheets up to central elements.
If $H$ is a spherical subgroup of $G$, by \cite[Theorem 1]{Brion}, there exists a flat deformation of $G/H$ to a homogeneous spherical space $G/H_0$, where $H_0$ contains a maximal unipotent subgroup of $G$: such an homogeneous space is called \emph{horospherical}, and $H_0$ a\emph{ horospherical contraction} of $H$, see also \cite{Vinberg}. Moreover, if $G/H$ is (isomorphic to) a conjugacy class, then $\C[G/H]\simeq_G \C[G/H_0]$, see \cite[Theorem 3.15]{Costantini2010}.

\begin{prop} \label{prop_horosph}
Let $G$ be a complex connected reductive algebraic group with $G'$ simply-connected.
Let $x_1, x_2 \in G_{sph}$. 
Then $\orb_{x_1}$ and $\orb_{x_2}$ are contained in the same birational sheet up to a central element if and only if $C_G(x_1)$ and $C_G(x_2)$ have the same horospherical contraction.
\end{prop}
\begin{proof}
Let  $x \in G_{sph}$ and $H=C_G(x)$.
We recall the description of  the horospherical contraction $H_0$ of $H$ containing $U$ from  \cite[Corollary 3.8]{Costantini2010}.
Let $w$ be the unique element in $W$ such that $\orb_x \cap BwB$ is dense in $\orb_x$.
By choosing $x\in wB$,  the dense $B$-orbit in $\orb_x$ is $\orb^B_x$.
Then $P:=\{g\in G\mid g\cdot \orb^B_x=\orb^B_x\}$ is a  parabolic subgroup containing $B$. Let $\Theta \subseteq \Delta$ be such that $P=P_\Theta$.
One has
$H_0=\langle{U^-,U_{w_\Theta},T_x}\rangle$, where, $w:=w_0w_{\Theta}$,
$U_{w_\Theta}:=U\cap L_\Theta$, $T_x:=T\cap C_G(x)$. 

We may assume that $x_i$ lies in the dense $B$-orbit $\orb^B_{x_i}$ \ ($\subseteq Bw_iB$), for $i=1, 2$.
We have seen that $\orb_{x_1}$ and $\orb_{x_1}$ are contained in the same birational sheet up to a central element  if and only if $\Lambda(\orb_{x_1})=\Lambda(\orb_{x_2})$.
The last  equality is equivalent to $w_1=w_2$ and $T_{x_1}=T_{x_2}$ by \cite[Lemma 3.9, Theorem 3.23]{Costantini2010}. 
\end{proof}

\begin{rk}
From the classification it follows that the birationally rigid unipotent conjugacy class $\orb^M$ appearing in the decomposition datum $\tau=(M, Z(M)^\circ z, \orb^M)$ is in fact rigid, except in the cases 
\begin{enumerate}
\item[(i)] $(G, \{1\},X_2)$ in type $\mathsf{C}_{n}$, $n\geq 3$;
\item[(ii)]  $(G,\{1\}, Z_m)$ in type $\mathsf{B}_{2m}$, $m\geq 2$.
\end{enumerate}
In the first (resp. second) case $\overline{J(\tau)}^{bir}$ is contained only in the (spherical) sheet corresponding to $(L_1, Z(L_1)^\circ, \{1\})$ (resp. $(L_n, Z(L_n)^\circ, \{1\})$).

In the other cases $\overline{J(\tau)}^{bir}$ is contained only in the sheet $\overline{J(\tau)}^{reg}$: in particular every spherical birational sheet is contained in a unique sheet.
\end{rk}

\section{Remarks for Lie algebras}
By \cite[Proposition 1]{Arzhantsev_1997}, the subset $\mathfrak{g}_{sph}$ consisting of spherical ajoint orbits is a union of sheets.
Since every birational sheet is contained in a sheet, $\mathfrak{g}_{sph}$ is a union of spherical birational sheets.
Being birational sheets disjoint, $\mathfrak{g}_{sph}$ is a disjoint union of spherical birational sheets.

Having described the spherical birational sheets in $G$, from the tables in \S \ref{ZB} one can easily deduce the corresponding classification of spherical birational sheets in $\lieg$.
In each table we have a spherical birational sheet $\overline{\mathfrak{J}(\mathfrak{l}, \Orb^L)}^{bir}$ for each $\tau=(L,Z(L)^\circ, \orb^L)$ with $L$ a Levi subgroup of $G$:
here $\mathfrak{l} := \Lie(L)$ and $ \Orb^L$ is the nilpotent orbit in $\mathfrak{l}$ corresponding to $\orb^L$.
Moreover, if $L=G$, then $\overline{\mathfrak{J}(\mathfrak{g}, \Orb)}^{bir}= \Orb$; 
if $L \lneq G$, then $L$ is a maximal Levi subgroup $L_i$ of $G$ and $\orb^{L_i}=\{1\}$ so that $\Orb^{L_i}=\{0\}$. 
There are two possibilities: either $\overline{J(\tau)}^{bir}$ does not contain unipotent conjugacy classes, or it contains a unique unipotent conjugacy class $\orb_u$.
In this case, let $\Orb_\nu$ be the corresponding nilpotent orbit in $\lieg$.
Then we have 
$\overline{\mathfrak{J}(\mathfrak{l_i}, \{0\})}^{bir}=\bigcup_{\zeta\not=0} 
\Orb_{\zeta \check{\omega}_i}$ in the first case and $\overline{\mathfrak{J}(\mathfrak{l_i}, \{0\})}^{bir}=\bigcup_{\zeta\not=0} 
\Orb_{\zeta \check{\omega}_i}\cup \Orb_\nu$ in the second case.
In particular this proves Losev's Conjecture \ref{conj} for $\mathfrak{g}_{sph}$:
\begin{thm}  \label{thm_principal_lie}
Let $\mathfrak{g}$ be reductive and let $\Orb_1$ and $\Orb_2$ be spherical adjoint orbits of $\mathfrak{g}$.
Then $\Lambda(\Orb_1) = \Lambda(\Orb_2)$ if and only if 
$\Orb_1$ and $\Orb_2$ are contained in the same birational sheet of $\mathfrak{g}$. \hfill \qed
\end{thm}

\section*{Acknowledgments}
The authors would like to thank Giovanna Carnovale and Bart Van Steirteghem for interesting conversations and Eric Sommers for his help in the proof of Theorem \ref{thm_sln}.
This research was partially supported by Project 34672 Parity Sheaves on Kashiwaras flag manifold, funded by the MIUR-DAAD Joint Mobility Program 2018/2019,
BIRD179758/17 Project Stratifications in algebraic groups, spherical varieties, Kac-Moody algebras and Kac-Moody groups funded by the University of Padova and
 by MIUR-Italy via PRIN ``Group theory and applications''.

\bibliography{biblio}

\begin{thebibliography}{10}

\bibitem{AF}
F.~Ambrosio.
\newblock Birational sheets in reductive groups.
\newblock To appear in \emph{Math. Z.} Online
  \href{https://doi.org/10.1007/s00209-020-02597-3}{DOI:10.1007/s00209-020-02597-3},
  2020.

\bibitem{Arzhantsev_1997}
I.~V. Arzhantsev.
\newblock Actions of the group {${\rm SL}_2$} that are of complexity one.
\newblock {\em Izv. Ross. Akad. Nauk Ser. Mat.}, 61(4):3--18, 1997.

\bibitem{Borho}
W.~Borho.
\newblock \"{U}ber {S}chichten halbeinfacher {L}ie-{A}lgebren.
\newblock {\em Invent. Math.}, 65(2):283--317, 1981/82.

\bibitem{BK79}
W.~Borho and H.~Kraft.
\newblock \"{U}ber {B}ahnen und deren {D}eformationen bei linearen {A}ktionen
  reduktiver {G}ruppen.
\newblock {\em Comment. Math. Helv.}, 54(1):61--104, 1979.

\bibitem{Bour4}
N.~Bourbaki.
\newblock {\em \'{E}l\'{e}ments de math\'{e}matique. {F}asc. {XXXIV}. {G}roupes
  et alg\`ebres de {L}ie. {C}hapitre {IV}: {G}roupes de {C}oxeter et syst\`emes
  de {T}its. {C}hapitre {V}: {G}roupes engendr\'{e}s par des r\'{e}flexions.
  {C}hapitre {VI}: syst\`emes de racines}.
\newblock Actualit\'{e}s Scientifiques et Industrielles, No. 1337. Hermann,
  Paris, 1968.

\bibitem{Brion}
M.~Brion.
\newblock Quelques propri\'{e}t\'{e}s des espaces homog\`enes sph\'{e}riques.
\newblock {\em Manuscripta Math.}, 55(2):191--198, 1986.

\bibitem{Broer}
A.~Broer.
\newblock The sum of generalized exponents and {C}hevalley's restriction
  theorem for modules of covariants.
\newblock {\em Indag. Math. (N.S.)}, 6(4):385--396, 1995.

\bibitem{broer_lectures}
A.~Broer.
\newblock Lectures on decomposition classes.
\newblock In {\em Representation theories and algebraic geometry ({M}ontreal,
  {PQ}, 1997)}, volume 514 of {\em NATO Adv. Sci. Inst. Ser. C Math. Phys.
  Sci.}, pages 39--83. Kluwer Acad. Publ., Dordrecht, 1998.

\bibitem{CarnoBul}
G.~Carnovale.
\newblock Lusztig's partition and sheets (with an appendix by {M}. {B}ulois).
\newblock {\em Math. Res. Lett.}, 22(3):645--664, 2015.

\bibitem{CE1}
G.~Carnovale and F.~Esposito.
\newblock On sheets of conjugacy classes in good characteristic.
\newblock {\em Int. Math. Res. Not. IMRN}, 2012(4):810--828, 2012.

\bibitem{Carter2}
R.~W. Carter.
\newblock {\em Finite groups of {L}ie type. Conjugacy classes and complex
  characters}.
\newblock Wiley Classics Library. John Wiley \& Sons, Ltd., Chichester, 1993.
\newblock Reprint of the 1985 original, A Wiley-Interscience Publication.

\bibitem{CollMcGov}
D.~H. Collingwood and W.~M. McGovern.
\newblock {\em Nilpotent orbits in semisimple {L}ie algebras}.
\newblock Van Nostrand Reinhold Mathematics Series. Van Nostrand Reinhold Co.,
  New York, 1993.

\bibitem{Costantini2010}
M.~Costantini.
\newblock On the coordinate ring of spherical conjugacy classes.
\newblock {\em Math. Z.}, 264(2):327--359, 2010.

\bibitem{Dix}
J.~Dixmier.
\newblock Polarisations dans les alg\`ebres de {L}ie.
\newblock {\em Ann. Sci. \'{E}c. {N}orm. {S}up\'er.}, 4e s{\'e}rie,
  4(3):321--335, 1971.

\bibitem{Fu2010}
B.~Fu.
\newblock On {$\mathbb{Q}$}-factorial terminalizations of nilpotent orbits.
\newblock {\em J. Math. Pures Appl. (9)}, 93(6):623--635, 2010.

\bibitem{Losev}
I.~Losev.
\newblock Deformations of symplectic singularities and orbit method for
  semisimple lie algebras.
\newblock Preprint
  \href{https://arxiv.org/abs/1605.00592v3}{arXiv:1605.00592v3}, 2020.

\bibitem{LusztigICC}
G.~Lusztig.
\newblock Intersection cohomology complexes on a reductive group.
\newblock {\em Invent. Math.}, 75(2):205--272, 1984.

\bibitem{LS79}
G.~Lusztig and N.~Spaltenstein.
\newblock Induced unipotent classes.
\newblock {\em J. London Math. Soc. (2)}, 19(1):41--52, 1979.

\bibitem{McGovern2002}
W.~M. McGovern.
\newblock The adjoint representation and the adjoint action.
\newblock In {\em Algebraic Quotients. Torus Actions and Cohomology. The
  Adjoint Representation and the Adjoint Action}, pages 159--238. Springer
  Berlin Heidelberg, Berlin, Heidelberg, 2002.

\bibitem{Moreau1}
A.~Moreau.
\newblock On the dimension of the sheets of a reductive {L}ie algebra.
\newblock {\em J. Lie Theory}, 18(3):671--696, 2008.

\bibitem{Moreau2}
A.~Moreau.
\newblock Corrigendum to ``{O}n the dimension of the sheets of a reductive
  {L}ie algebra'' [mr2493061].
\newblock {\em J. Lie Theory}, 23(4):1075--1083, 2013.

\bibitem{Nami2009}
Y.~Namikawa.
\newblock Induced nilpotent orbits and birational geometry.
\newblock {\em Adv. Math.}, 222(2):547--564, 2009.

\bibitem{SommersBC}
E.~Sommers.
\newblock A generalization of the {B}ala-{C}arter theorem for nilpotent orbits.
\newblock {\em Internat. Math. Res. Notices}, 1998(11):539--562, 1998.

\bibitem{Vinberg}
E.~B. Vinberg.
\newblock Complexity of actions of reductive groups.
\newblock {\em Funktsional. Anal. i Prilozhen.}, 20(1):1--13, 96, 1986.

\end{thebibliography}

\end{document}